\documentclass[a4, 12pt,
]{amsart}

\usepackage{comment}

\usepackage[scr=rsfs,cal=euler]{mathalfa}

\usepackage{ulem} 
\usepackage[dvipdfmx]{graphicx,xcolor} 
\usepackage[dvipdfmx]{hyperref} 
\usepackage[backrefs, alphabetic, initials]{amsrefs}

\usepackage{tikz}
\usepackage[all]{xy}
\usepackage{tikz-cd}

\usepackage{color}
\usepackage{xcolor}
\usepackage{amssymb, bm, bbm}
\usepackage{hyperref}
\hypersetup{
	colorlinks=true,
	linkcolor=red,
	citecolor=magenta}

\newtheorem{theo}{Theorem}[section]
\newtheorem{prop}[theo]{Proposition}
\newtheorem{lemm}[theo]{Lemma}
\newtheorem{cor}[theo]{Corollary}
\newtheorem{claim}[theo]{Claim}

\newtheorem{conj}[theo]{Conjecture}

\numberwithin{equation}{section}

\theoremstyle{definition}
\newtheorem{defi}[theo]{Definition}
\newtheorem{ex}[theo]{Example}

\newtheorem{setup}[theo]{Setting}

\theoremstyle{remark}
\newtheorem{rem}[theo]{Remark}

\newcommand{\Supp}[0]{\operatorname{Supp}}

\newcommand{\rank}[0]{\operatorname{rank}}
\newcommand{\codim}[0]{\operatorname{codim}}
\newcommand{\Sym}[0]{\operatorname{Sym}}

\newcommand{\ddbar}{dd^c}

\newcommand{\e}{\varepsilon}

\newcommand{\OX}{\mathcal{O}}

\newcommand{\reg}{{\rm{reg}}}
\newcommand{\sing}{{\rm{sing}}}
\newcommand{\pr}{{\rm{pr}}}
\newcommand{\id}{{\rm{id}}}

\newcommand{\cF}{\mathcal{F}}

\newcommand{\ZZ}{\mathbb{Z}}
\newcommand{\CC}{\mathbb{C}}
\newcommand{\QQ}{\mathbb{Q}}
\newcommand{\RR}{\mathbb{R}}
\newcommand{\PP}{\mathbb{P}}

\newcommand{\cal}[1]{\mathcal{#1}}

\makeatletter
\newcommand*{\da@rightarrow}{\mathchar"0\hexnumber@\symAMSa 4B }
\newcommand*{\da@leftarrow}{\mathchar"0\hexnumber@\symAMSa 4C }
\newcommand*{\xdashrightarrow}[2][]{%
  \mathrel{%
    \mathpalette{\da@xarrow{#1}{#2}{}\da@rightarrow{\,}{}}{}%
  }%
}
\newcommand{\xdashleftarrow}[2][]{%
  \mathrel{%
    \mathpalette{\da@xarrow{#1}{#2}\da@leftarrow{}{}{\,}}{}%
  }%
}
\newcommand*{\da@xarrow}[7]{%
  \sbox0{$\ifx#7\scriptstyle\scriptscriptstyle\else\scriptstyle\fi#5#1#6\m@th$}%
  \sbox2{$\ifx#7\scriptstyle\scriptscriptstyle\else\scriptstyle\fi#5#2#6\m@th$}%
  \sbox4{$#7\dabar@\m@th$}%
  \dimen@=\wd0 %
  \ifdim\wd2 >\dimen@
    \dimen@=\wd2 %
  \fi
  \count@=2 %
  \def\da@bars{\dabar@\dabar@}%
  \@whiledim\count@\wd4<\dimen@\do{%
    \advance\count@\@ne
    \expandafter\def\expandafter\da@bars\expandafter{%
      \da@bars
      \dabar@
    }%
  }%
  \mathrel{#3}%
  \mathrel{%
    \mathop{\da@bars}\limits
    \ifx\\#1\\%
    \else
      _{\copy0}%
    \fi
    \ifx\\#2\\%
    \else
      ^{\copy2}%
    \fi
  }%
  \mathrel{#4}%
}
\makeatother
\usepackage{geometry}
\begin{document}

\title[Compact K\"ahler three-folds with nef anti-canonical bundle ]
{Compact K\"ahler three-folds \\ with nef anti-canonical bundle }

\author{Shin-ichi MATSUMURA}
\address{Mathematical Institute, Tohoku University,
6-3, Aramaki Aza-Aoba, Aoba-ku, Sendai 980-8578, Japan.}
\email{{\tt mshinichi-math@tohoku.ac.jp}}
\email{{\tt mshinichi0@gmail.com}}

\author{Xiaojun WU}
\address{Laboratoire Math\'ematiques et Interactions J.A. Dieudonn\'e,
Universit\'e C$\rm{\check{o}}$te d'Azur, Nice, 06108, France}
\email{{\tt xiaojun.wu@uni-bayreuth.de}}
\email{\tt xiaojun.wu@univ-cotedazur.fr}{}

\date{\today, version 0.01}

\renewcommand{\subjclassname}{%
\textup{2020} Mathematics Subject Classification}


\maketitle

\begin{abstract}
In this paper, we prove that a non-projective compact K\"ahler three-fold with nef anti-canonical bundle is,
up to a finite \'etale cover, one of the following:
a manifold with vanishing first Chern class; the product of a K3 surface and the projective line; or a projective space bundle over a $2$-dimensional torus.
This result extends Cao-H\"oring's structure theorem for projective manifolds
to compact K\"ahler manifolds in dimension $3$.
For the proof, we investigate the Minimal Model Program for compact K\"ahler three-folds with nef anti-canonical bundles by using the positivity of direct image sheaves, $\mathbb{Q}$-conic bundles, and orbifold vector bundles.
\end{abstract}

\tableofcontents

\section{Introduction}\label{Sec1}

\subsection{Background and the main results}\label{subsec1-1}

In this paper, we study a structure theorem for compact K\"ahler manifolds with nef anti-canonical bundle,
motivated by the conjecture below.
This conjecture serves as a natural extension of the pioneering studies
for manifolds with nef tangent bundles \cite{DPS94}  and manifolds with non-negative holomorphic bisectional curvatures  \cite{HSW81, Mok88}.
The studies of manifolds with nef anti-canonical bundles  have advanced through
interactions with other studies of  ``non-negatively curved'' varieties
(e.g.,\,see \cite{Mat20, Mat22, Mat22b, HIM22} for recent developments).

\begin{conj}\label{conj-main}
Let $X$ be a compact K\"ahler manifold with the nef anti-canonical bundle  $-K_{X}$.
Then, there exists a fibration $\varphi\colon  X \to Y$ with the following$:$
\begin{itemize}
\item $\varphi\colon  X \to Y$ is a locally constant fibration$;$
\item $Y$ is a compact K\"ahler manifold with $c_{1}(Y)=0$$;$
\item $F$, which is the fiber of $\varphi\colon  X \to Y$,  is rationally connected.
\end{itemize}
\end{conj}

The notion of locally constant fibrations (e.g.,\,see \cite[Definition 2.3]{MW})
is stronger than that of locally trivial fibrations.
However, in this paper, readers unfamiliar with locally constant fibrations may regard them as locally trivial fibrations,
except for  Proposition \ref{prop-main}.

Conjecture \ref{conj-main} has been addressed by the theory of holonomy groups \cite{CDP15, DPS96}
under the stronger assumption that $-K_{X}$ is semi-positive
(i.e.,\,it admits a smooth Hermitian metric with semi-positive curvature).
Nevertheless, the nef case presents significantly greater challenges than semi-positive case,
just as it was highly non-trivial to generalize the structure theorem from manifolds with non-negative holomorphic bisectional curvatures to those with nef tangent bundles.
In the case where $X$ is a projective manifold, Cao-H\"oring solved the conjecture in \cite{Cao19, CH19}
(see \cite{CCM21, MW, Wan20} for projective klt pairs),
but their proofs require an ample line bundle on $X$;
therefore, we cannot at least directly apply their proofs to compact K\"ahler manifolds.

This paper aims to solve Conjecture \ref{conj-main} in the case of $\dim X = 3$ (see Theorem \ref{theo-main}).
Theorem \ref{theo-main} follows directly from Theorem \ref{theo-non-proj},
as detailed in Proposition \ref{prop-main}.
Thus, this paper focuses on proving Theorem \ref{theo-non-proj}.

\begin{theo}\label{theo-main}
Conjecture \ref{conj-main} is true in the case of $\dim X = 3$.
\end{theo}

\begin{theo}\label{theo-non-proj}
Let $X$ be a non-projective compact K\"ahler three-fold with nef anti-canonical bundle.
Then $X$ admits a finite \'etale cover that is one of the following$:$
\begin{itemize}
 \item a  compact K\"ahler manifold with vanishing first Chern class$;$
 \item the product of a  K3 surface and the projective line $\PP^1$$;$
 \item the projective space bundle $\PP(\mathcal{F})$ of a numerically flat vector bundle $\mathcal{F}$ of rank $2$
 over a $2$-dimensional $($compact complex$)$ torus.
\end{itemize}
\end{theo}

\subsection{Strategy of the proof of Theorem \ref{theo-non-proj}}\label{subsec1-2}

In this subsection, we outline the proof of  Theorem \ref{theo-non-proj}.
Let $X$ be a non-projective compact K\"ahler three-fold with nef anti-canonical bundle
and let $\varphi\colon   X \dashrightarrow R(X)$ be an MRC (maximally rationally connected) fibration of $X$
(see \cite{KoMM92, Cam92} for MRC fibrations).
The candidates of $X$ in Theorem \ref{theo-non-proj} are determined by $\varphi\colon   X \dashrightarrow R(X)$.

We initially verify that it suffices to consider the case of $\dim R(X)=2$.
In the case of $\dim R(X)=0$, the manifold $X$ is rationally connected, and hence projective.
In the case of $\dim R(X)=1$, a rationally connected fiber $F$ of $\varphi\colon   X \dashrightarrow R(X)$
has no non-zero holomorphic differential forms.
As a result, we obtain $h^{2}(X, \mathcal{O}_{X})=h^{0}(X, \Omega_{X}^{2})=0$ from $\dim R(X)=1$,
implying that $X$ is projective.
In the case of $\dim R(X)=3$, the manifold $X$ is non-uniruled;
hence $K_X$ is pseudo-effective.
This follows from \cite{BDPP} for projective manifolds of any dimension
and from \cite[Corollary 1.2]{Bru06}  for compact K\"ahler manifolds of dimension $\leq 3$.
Consequently, the nefness of $-K_X$ implies that $c_{1}(X)=c_{1}(K_{X})=0$.

We now revisit  Cao-H\"oring's proof \cite{CH19}, which shows
that a projective manifold $X$ admits a locally constant MRC fibration.
For simplicity, we suppose that $\varphi\colon  X \dashrightarrow R(X)$ is
an everywhere-defined holomorphic map onto a smooth projective variety $R(X)$.
The essence of the proof involves constructing a $\varphi$-ample line bundle  $B$ on $X$
such that the direct image sheaf $\varphi_{*}(pB):=\varphi_{*}\mathcal{O}_{X}(pB)$
is weakly positively curved
and satisfies that $c_{1}(\varphi_{*}(pB))=0$ for $1 \ll p \in \mathbb{Z}$
(see Subsections \ref{subsec2-1} and \ref{subsec4-1} for details).
Then, Simpson's result \cite{Sim92} confirms that $\varphi_{*}(pB)$ admits a flat connection,
which implies that $X \to R(X)$ is a locally constant fibration.

Now, let us return to the case where $X$ is non-projective.
A  significant  challenge arises in this context: $X$ might not have a $\varphi$-ample line bundle,
even if $\varphi\colon X \dashrightarrow R(X)$ is a holomorphic map.
Our idea to overcome this challenge is to apply  the Minimal Model Program (MMP) 
for compact K\"ahler three-folds,
as developed in \cite{HP15a, HP15b, HP16}.
Note that  the MMP approach has been previously treated in  \cite{BP04, PS98}.
By running the MMP, we can find $X \dashrightarrow X' \to Z$,
where  $X \dashrightarrow X'$ is a composition of divisorial contractions and flips
and $\varphi\colon  X' \to S$ is an MFS (Mori fiber space) (see Theorem  \ref{3-fold-MMP} for details).
An advantage of running the MMP is that $-K_{X'}$ is $\varphi$-ample by construction;
thus, we can expect that Cao-H\"oring's proof works for $\varphi\colon  X' \to S$.

In considering $\varphi\colon  X' \to S$ instead of $\varphi\colon  X \dashrightarrow R(X)$,
we face the new difficulties compared to Cao-H\"oring's argument:
The first difficulty is that $-K_{X'}$ is not necessarily nef.
To overcome this difficulty, based on an observation in \cite{EIM},
we focus on the fact that the non-nef locus of $-K_{X'}$ is not dominant over $S$
(see Subsection \ref{subsec3-3}),
which enables us to treat our situation as in the case where $-K_{X'}$ is nef.
Thus, we can construct a $\varphi$-ample line bundle $B$ on $X'$
such that $\varphi_{*}(pB)$ is  weakly positively curved
and satisfies that $c_{1}(\varphi_{*}(pB))=0$.
The second difficulty is that $S$ may have singularities,
which prevents us from obtaining a flat connection on $\varphi_{*}(pB)$.
To overcome this difficulty, we observe that $\varphi\colon  X' \to S$ is a toroidal $\mathbb{Q}$-conic bundle
and $\varphi_{*}(pB)$ is an orbifold vector bundle on $S$,
which enables us to obtain a flat connection.

To implement the above ideas, we actually need to divide the situation into two cases. 
In Subsection \ref{subsec4-2}, we consider the case where $X$ is not simply connected.
In this case, we focus on the Albanese map $\alpha\colon  X \to A(X)$ after taking a finite \'etale cover of $X$
(cf.\,\cite{NW}).
Each step in the MMP contracts rational curves and $A(X)$ has no rational curve;
hence we can find a morphism $\beta\colon S \to A(X)$.
Then, comparing  $S$ to $A(X)$,
we show that the MFS $\varphi\colon  X' \to S$ is actually a conic bundle.
Then, by studying conic bundles in the non-projective setting,
we deduce that $\varphi\colon  X' \to S$ is a projective space bundle and $X \dashrightarrow X'$ is isomorphic.
In  Subsection \ref{subsec4-3}, we treat the remaining case where $X$ is simply connected.
In this case, by applying the theory of orbifold vector bundles,
we show that after taking the base change by an appropriate finite quasi-\'etale cover,
the MFS $\varphi \colon  X' \to S$ is a locally constant fibration,
which implies that $X \dashrightarrow X'$ is isomorphic.

\subsection*{Notation and Conventions}\label{subsec-not}

We use the terms ``Cartier divisors,'' ``invertible sheaves,''  and ``line bundles'' interchangeably,
and adopt the additive notation for tensor products
(e.g.,\,$L+M:=L\otimes M$ for line bundles $L$ and $M$).
Additionally, we  use the terms  ``locally free sheaves'' and ``vector bundles'' interchangeably,
and often refer to singular Hermitian metrics simply as ``metrics.''
The term ``fibrations'' refers to a proper surjective morphism with connected fibers,
the term ``analytic varieties'' refers to an irreducible and reduced complex analytic space,
the term ``K\"ahler spaces'' refers to a normal analytic variety admitting a K\"ahler form
(i.e.,\,a smooth positive $(1,1)$-form on $X$ with local potential).

\subsection*{Acknowledgment}\label{subsec-ack}
The authors would like to thank Prof.\,Patrick Graf for providing a simpler proof of Lemma \ref{lem-topo}
than the original one in a draft of this paper.
They also express their gratitude to an anonymous referee who carefully reviewed the draft,
offered numerous constructive suggestions, and pointed out an error in Subsection \ref{subsec4-3} of the draft.
The first author was partially  supported
by Grant-in-Aid for Scientific Research (B) $\sharp$21H00976
and Fostering Joint International Research (A) $\sharp$19KK0342 from Japan Society for the Promotion of Science.
The second author was supported by
DFG Projekt ``Singul\"are hermitianische Metriken f\"ur Vektorb\"undel und Erweiterung kanonischer Abschnitte,''  managed by Prof.\,Mihai P\u{a}un.

\section{Preliminary Results}\label{Sec2}

\subsection{Bott-Chern cohomology groups on normal analytic varieties}\label{subsec2-2}

In this subsection, following \cite{BEG},
we review Bott-Chern cohomology groups and positive currents on normal analytic varieties.

Let $X$ be a normal analytic variety.
A pluriharmonic function on  $X$ can be locally written as the real part of a holomorphic function,
in other words, the kernel of the $\partial \overline{\partial}$-operator acting on the sheaf of distributions of bidegree $(0,0)$
coincides with the sheaf $\RR \OX_X$ of real parts of holomorphic functions
(e.g.,\,see \cite[Lemma 4.6.1]{BEG}).
Then, the {\textit{Bott-Chern cohomology group}} of $X$ is defined by
$$
H^{1,1}_{BC}(X, \CC):=H^1(X, \RR \OX_X).
$$
The first Chern class $c_1(L) \in H^{1,1}_{BC}(X, \CC)$ of a line bundle $L$ on $X$  is defined
by the Bott-Chern cohomology class of $(\sqrt{-1}/2\pi )\Theta_{h}(L)$,
where $\Theta_{h}(L)$ denotes the Chern curvature of a smooth Hermitian metric $h$ on $L$.
Note that $c_1(L)$ does not depend on the choice of smooth Hermitian metrics and
the first Chern class of $\QQ$-Cartier divisors can also be defined by linearity.
There exists the natural morphism $H^{1,1}_{BC}(X, \CC) \to H^2(X, \RR)$
such that the first Chern class of line bundles coincides with the topological definition,
where $H^2(X, \RR)$ denotes the singular cohomology of $X$ (e.g.,\,see \cite[Page 230]{BEG}).

The proposition below is often used as an extension theorem for positive currents
representing Bott-Chern cohomology classes (see \cite{Dem85} for currents on analytic varieties).
The lemma below is a generalization of the support theorem to analytic varieties.

\begin{prop}[{\cite[Proposition 4.6.3]{BEG}}] \label{BEG}
Let $\alpha \in H^{1,1}_{BC}(X, \CC)$ be a Bott-Chern cohomology class on a normal analytic variety $X$, and let $T$
be a positive current on $X_\reg$ representing the restricted class $\alpha|_{X_\reg} \in H^{1,1}_{BC}(X_{\reg}, \CC)$.
Then, the current $T$ is uniquely extended to the positive current with local potential on $X$ representing $\alpha \in H^{1,1}_{BC}(X, \CC)$.
\end{prop}

\begin{lemm} \label{support-lemma}
Let $X$ be an analytic variety, and let $T_{1}$, $T_{2}$ be $d$-closed positive currents of bidimension $(p,p)$
$($without assuming that they admit local potentials$)$.
If the support of the difference $T:=T_{1}-T_{2}$ is contained
in a Zariski closed subset  $A \subset X$ of dimension $< p$, then we have $T = 0$.
\end{lemm}
\begin{proof}[Proof of Lemma \ref{support-lemma}]
The statement is local in $X$;
therefore, we may assume that there exists an embedding $i\colon  X \to B \subset \CC^N$ of $X$ into an open set $B \subset \CC^N $.
Since the pushforward $i_* T$ is a normal current (e.g.,\,see \cite[Chap.\,III, Section 2.C]{agbook} for the definition of normal currents),
the support theorem for smooth varieties shows that $i_* T=0$,
which implies that $T=0$.
\end{proof}

\subsection{Positivity of sheaves on normal analytic varieties}\label{subsec2-1}
In this subsection, following \cite{HPS18, PT18, Mat22},
we briefly review singular Hermitian metrics on torsion-free sheaves on normal analytic varieties.

Let $\cal{E}$ be a  torsion-free  coherent sheaf on a normal analytic variety $X$.
\textit{A singular Hermitian metric} $h$ on $\cal{E}$
is a possibly singular Hermitian  metric on the vector bundle $\mathcal{E}|_{X_{0}}$
(see \cite[Definition 17.1]{HPS18} and \cite[Definition 2.2.1]{PT18} for metrics on vector bundles).
Here $\mathcal{E}|_{X_{0}}$ is the restriction of $\mathcal{E}$ to $X_{0}:=X_{\reg} \cap X_\mathcal{E}$,
where  $X_{\reg}$ is the non-singular locus of $X$
and $X_\mathcal{E}$ is the maximal locally free locus of $\mathcal{E}$.
Note that $X_{0} \subset X$ is a Zariski open set with $\codim (X \setminus X_{0})\geq 2$.
For a smooth $(1,1)$-form $\theta$ on $X$ with local potential,
we write
$$
\sqrt{-1}\Theta_{h} \geq \theta \otimes {\rm{id}} \text{ on } X
$$
if  the function $\log |e|_{h^{*}}-f$ is psh for any local section $e $ of $\mathcal{E}^{*}$,
where $f$ is a local potential of $\theta$ (i.e.,\,$\theta=\sqrt{-1} \partial \overline{\partial} f$) and
$h^{*}$ is the induced metric on the dual sheaf
$\mathcal{E}^*:=\mathit{Hom} (\mathcal{E}, \mathcal{O}_{X})$.
The plurisubharmonicity can be extended through a Zariski closed set of codimension $\geq 2$;
therefore it suffices to check the plurisubharmonicity on an open set of $X_{0}$.

\begin{defi}\label{def-psef}
Let $X$ be a K\"ahler space with a  K\"ahler form $\omega_{X}$
and $\theta$ be a $(1,1)$-form on $X$ with local potential.
Let $\mathcal{E}$  be a torsion-free sheaf on $X$
\begin{itemize}
\item[$(1)$] $\mathcal{E}$ is said to be {\textit{$\theta$-positively curved}}
if $\mathcal{E}$ admits a singular Hermitian metric $h$ such that
$
\sqrt{-1}\Theta_{h} \geq \theta \otimes {\rm{id}} \text{ on } X.
$
\item[$(2)$]
$\mathcal{E}$  is said to be {\textit{$\theta$-weakly positively curved}}
if $\cal{E}$ admits singular Hermitian metrics $\{h_{\e}\}_{\e>0}$ 
such that $\sqrt{-1}\Theta_{h} \geq (\theta - \e \omega_{X}) \otimes {\rm{id}}$ on  $X$.
\item[$(3)$] $\mathcal{E}$ is simply said to be {\textit{positively curved}}
or {\textit{weakly positive curved}}  in the case of $\theta=0$.
\end{itemize}
\end{defi}

When $X$ is compact, the notion of weakly positively curved sheaves
does not depend on the choice of $\omega_{X}$ and
is stronger than the notion of pseudo-effective sheaves in the sense of \cite[Definition 2.1]{Mat22c}.

\subsection{On direct image sheaves for projective morphisms}\label{subsec4-1}

This subsection aims to prove Theorem \ref{theo-flat}.
For this purpose, we prepare the following proposition:

\begin{prop}[{cf.\,\cite[2.8 Proposition]{CH19}, \cite[Theorem 2.2 (1)]{CCM21}}]\label{prop-direct}
Let $\varphi\colon X\to Y$ be a fibration between $($not necessarily compact$)$ K\"ahler manifolds $X$ and $Y$
with the K\"ahler forms $\omega_{X}$ and $\omega_{Y}$.
Let $L$ be a line bundle on $X$ and $\theta$ be a $(1,1)$-form on $Y$ with local potential.
Assume the following conditions$:$
\begin{itemize}
\item[(a)]
The non-nef locus of $-K_{X/Y}$ is not dominant over $Y$ in the following sense$:$
$- K_{X/Y}$ has singular Hermitian metrics $\{g_{\delta}\}_{\delta>0}$ such that
$\sqrt{-1}\Theta_{g_{\delta}} \geq -\delta \omega _{X}$ holds on $X$
and  $\{x \in X\,|\, \nu(g_{\delta}, x)>0 \}$
is not dominant over $Y$,
where $\nu(g_{\delta}, x)$ is the Lelong number of a local potential of $g_{\delta}$ at $x$$;$

\item[(b)]
$L$ is a $\varphi$-big line bundle in the following sense$:$
$L$ has a singular Hermitian metric $g$ such that
$\sqrt{-1}\Theta_{g}+\varphi^{*}\omega_{Y} \geq \omega_{X}$ holds on $X$$;$

\item[(c)]
$L$ is $\varphi^{*}\theta$-weakly positively curved in the following sense$:$
$L$ has singular Hermitian metrics $\{h_{\delta'}\}_{\delta'>0}$
such that $\sqrt{-1}\Theta_{h_{\delta'}}\geq \varphi^{*}\theta - \delta' \omega_{X}$ on $X$.
\end{itemize}
Then, we have$:$
\begin{itemize}
\item[$(1)$]  The direct image sheaf $\varphi_{*}(-mK_{X/Y} +L)$ is
$((1-\varepsilon ) \theta-\varepsilon\omega_{Y})$-positively curved
for any $m\in \mathbb{Z}_{+}$ and $\e>0$.

\item[$(2)$] If we further assume that $\omega_{Y} \geq \theta$ holds,
then $\varphi_{*}(-mK_{X/Y} +L)$ is $\theta$-weakly positively curved.
In particular, if $L$ is a pseudo-effective line bundle $($i.e.,\,positively curved$)$,
then $\varphi_{*}(-mK_{X/Y} +L)$ is weakly positively curved.
\end{itemize}
\end{prop}
\begin{rem}
Throughout this paper, for a line bundle $M$ on $X$,  the notation  $\varphi_{*}(M)$ denotes
the direct image sheaf  $\varphi_{*}(\mathcal{O}_{X}(M))$ of the invertible sheaf  $\mathcal{O}_{X}(M)$.
\end{rem}

\begin{proof}
We apply the positivity of direct image sheaves \cite[Theorem 5.1.2]{PT18} (see also \cite{HPS18}) to
construct the desired singular Hermitian metrics on $\mathcal{W}_{m}:=\varphi_{*} (-mK_{X/Y} +L)$.
Note that \cite[Theorem 5.1.2]{PT18}  is stated for projective fibrations,
but in fact it is valid for K\"ahler fibrations (see  \cite[Proposition 2.5, Theorem 2.6]{Wan21}).
To this end, we consider the following decomposition of  $-mK_{X/Y} + L $:
\begin{align*}
-mK_{X/Y} + L &= k K_{X/Y}
\overbrace{-(m+k)K_{X/Y}}^{\textup{with  $g_{\delta}^{m+k}$}}
+ \overbrace{L.}^{\textup{with  $g^{\e} \cdot h_{\delta'}^{1-\e}$}}
\end{align*}
Furthermore, we define the metric $G$ on $-(m+k)K_{X/Y}+L$
by
$$ G := g_{\delta}^{m+k} \cdot g^{\e} \cdot h_{\delta'}^{1-\e}.
$$
We can easily confirm that the multiplier ideal sheaf $\mathcal{I}(G^{1/k})$ satisfies that
\begin{align}\label{eq-mul}
\mathcal{I}(G^{1/k})|_{X_{y}} =
\mathcal{I}\big( g_{\delta}^{1+(m/k)}  \cdot g^{\e/k} \cdot h_{\delta'}^{(1-\e)/k} \big)|_{X_{y}}=\mathcal{O}_{X_{y}}
\end{align}
for a very general fiber $X_{y}$ and  a sufficiently large $k \gg 1$.
Indeed,  since Condition (a) implies that
$$
\nu(g_{\delta}^{1+(m/k)}  \cdot g^{\e/k} \cdot h_{\delta'}^{(1-\e)/k}, x) \leq \nu(g^{1/k} \cdot h_{\delta'}^{1/k}, x) < 1 \text{ for any $x \in X_{y}$},
$$
we see that
$$
\mathcal{O}_{X_{y}} = \mathcal{I}(G^{1/k}|_{X_{y}})\subset
\mathcal{I}(G^{1/k})|_{X_{y}} \subset \mathcal{O}_{X_{y}}
$$
by Skoda's lemma and the restriction formula.
This indicates that the support of $\mathcal{O}_{X}/\mathcal{I}(G^{1/k})$ is not dominant over $Y$. 
Hence, the natural inclusion
\begin{align}\label{eq-gensur}
\varphi_{*}((-mK_{X/Y} +L)\otimes \mathcal{I}(G^{1/k})) \to \varphi_{*}(-mK_{X/Y} +L)
\end{align}
is generically surjective.
Meanwhile, by the construction of metrics, we can easily see that
\begin{align*}
\sqrt{-1}\Theta_{G} &\geq - \delta (m+k) \omega_{X} + \e \omega_{X} - \e \varphi^{*} \omega_{Y} +(1-\e) \varphi^{*} \theta
- (1-\e) \delta' \omega_{X}\\
& \geq \big( \e -\delta(m+k)  -\delta'\big) \omega_{X} - \e \varphi^{*}\omega_{Y} +(1-\e) \varphi^{*}  \theta.
\end{align*}
For a given  $\e>0$,
after taking $\delta'=\delta'(\e)>0$ with $\delta'<(1/2)\e$,
we fix a sufficiently large $k=k(\delta')=k(\e)$ satisfying \eqref{eq-mul}.
Furthermore, we  take $1\gg \delta=\delta(m,k,\e) >0$ so that $\delta(m+k)<(1/2)\e$.
Then, the right-hand side is bounded from below
by $\varphi^{*}(- \e\omega_{Y} +(1-\e)\theta)$.
From this curvature estimate and \eqref{eq-gensur},
we can deduce that the sheaf $\varphi_{*}(-mK_{X/Y} +L)$ has the desired singular Hermitian metrics in Conclusion (1),
by applying the positivity of direct image sheaves
(see \cite[Theorem 5.1.2]{PT18} or \cite[Proposition 2.5, Theorem 2.6]{Wan21})
to $G e^{-\varphi^{*} f}$, where $f$ is a local potential of $- \e\omega_{Y} +(1-\e)\theta$.
Conclusion (2) directly follows from $ -\varepsilon\omega_{Y}  +(1-\varepsilon ) \theta \geq -2\varepsilon\omega_{Y}+ \theta$.
\end{proof}

\begin{theo}\label{theo-flat}
Let $\varphi\colon X\to Y$ be an equi-dimensional fibration between compact K\"ahler spaces $X$ and $Y$.
Let  $Y_{0} \subset Y$ be  a Zariski open set  with $\codim (Y \setminus Y_{0}) \geq 2$
such that $X_{0}:=\varphi^{-1}(Y_{0})$ and $Y_{0}$ are smooth
and that $\varphi_{0}:=\varphi|_{X_{0}}\colon X_{0} \to Y_{0}$ is a smooth fibration.
Let $L$ be a line bundle on $X$.
Assume the following conditions$:$
\begin{itemize}
\item[(a)] $-K_{X}$ is $\QQ$-Cartier and
the non-nef locus of $-K_{X}$ is not dominant over $Y$ in the sense of
Proposition \ref{prop-direct} $($\rm{a}$)$$;$
\item[(b)] $-K_{Y}$ is $\QQ$-Cartier and numerically trivial$;$
\item[(c)] $L$ is a pseudo-effective (i.e.,\,positively curved) and $\varphi$-ample line bundle on $X$$;$
\item[(d)] For any $p \in \mathbb{Z}_{+}$ with $\varphi_{*}(pL) \not =0$,
the line bundle $\det (\varphi_{*}(pL)) |_{Y_{0}}$ has a smooth Hermitian metric $g_{p}$
such that $\eta_{p}:=\sqrt{-1}\Theta_{g_{p}} \geq  -\omega_{Y}$ holds on $Y_{0}$
for some K\"ahler form $\omega_{Y}$.
\end{itemize}
Let $r$ be the rank of $\varphi_{*}(L)$
and $p$ be a sufficiently large integer with $p/r \in \mathbb{Z}_{+}$.
Define the sheaf $\mathcal{V}_{p}$ on $Y$ by
$$
\mathcal{V}_{p}:=\varphi_{*}(pL)\otimes
\Big( \dfrac{p}{r}\det \varphi_{*}(L) \Big)^{*}.
$$
Then, both $\mathcal{V}_{p}$ and $(\det \mathcal{V}_{p})^{*}$ are weakly positively curved.
\end{theo}
\begin{rem}\label{rem-flat}
The determinant sheaf $\det \varphi_{*}(L):= (\Lambda^{r} \varphi_{*}(L))^{**}$
is a reflexive sheaf of rank $1$, but not necessarily invertible when $Y$ has singularities;
therefore, precisely speaking, the notation $(p/r)\det \varphi_{*}(L) $ should be replaced with
$((\det \varphi_{*}(L))^{\otimes (p/r)})^{**}$.
Nevertheless, we mainly handle only  $\det \varphi_{*}(L)|_{Y_{0}}$,
which  is a line bundle on $Y_{0}$; hence, this notation does not cause confusion.

We apply this theorem to the case where the sheaf $\varphi_{*}(pL)$ is an orbifold vector bundle.
Condition (d), which may appear as a technical requirement, is automatically satisfied in this case.
\end{rem}
\begin{proof}
For the proof, employing  $L$ instead of ample line bundles, we adopt the argument in \cite{Cao19, CH19}.
However, the original argument in \cite{Cao19, CH19} is not so easy,
and our proof requires solving several technical issues.
Therefore, we write a detailed proof for the reader's convenience.

We may assume that $\varphi_{*}(pL)$ is locally free on $Y_{0}$
by removing the non-locally free locus of $\varphi_{*}(pL)$ from $Y_{0}$.
We initially prove Claim \ref{claim1} and Claim \ref{claim2}.

\begin{claim}\label{claim1}
$\varphi_{*}(pL)$ is weakly positively curved on $Y$ for any $p \in \mathbb{Z}_{+}$.
\end{claim}

\begin{proof}
The assumptions of Proposition \ref{prop-direct} for $\theta=0$
are satisfied from Conditions (a), (b), (c) of Theorem \ref{theo-flat}.
Hence, by applying Proposition \ref{prop-direct} to $\varphi_{0}=\varphi|_{X_{0}}\colon X_{0} \to Y_{0}$,
we can obtain singular Hermitian metrics $\{H_{\e}\}_{\e>0}$ on $\varphi_{*}(pL) |_{Y_{0}}$
such that $\sqrt{-1}\Theta_{H_{\e}} \geq  -\varepsilon \omega_{Y} \otimes \id$ holds on $Y_{0}$
for some K\"ahler form $\omega_{Y}$.
Then, by $\codim (Y \setminus Y_{0}) \geq 2$,
the metrics $H_{\e}$ can be automatically extended to $Y$,
where we implicitly used that
$\omega_{Y}$ is a K\"ahler form defined on $Y$ (not only on $Y_{0}$).
This means that $\varphi_{*}(pL)$ is weakly positively curved on $Y$.
\end{proof}

\begin{claim}\label{claim2}
Let $r_{p}$ be the rank of $\varphi_{*}(pL)$.
Then, the sheaf
$$
r_{p}pL \otimes \big( \varphi^{*}\det \varphi_{*}(pL) \big)^{*}
$$
is weakly positively curved on $X$.
\end{claim}
\begin{proof}
The basic strategy is the same as in \cite[Proposition 3.15]{Cao19},
but some different arguments are required because of the lack of ample line bundles.
For simplicity of the notation, we assume that
$p=1$ by replacing $L$ with $pL$. 
Fix a K\"ahler form $\omega_{Y}$ on $Y$ with Condition (d) of Theorem \ref{theo-flat}.
We may assume that 
$$
\eta_{1}+(1/r)\omega_{Y}
=\sqrt{-1}\Theta_{g_{1}} + (1/r)\omega_{Y}$$ 
is a K\"ahler form 
by replacing $\omega_{Y}$ with $k\omega_{Y}$, 
where $r:=r_{1}$ and $k \gg 1$.
Furthermore, we take a K\"ahler form $\omega_{X}$ on $X$ such that
$\omega_{X} \geq   \varphi^{*}\omega_{Y}$.
Note that this condition is preserved 
by replacing  $\omega_{X}$ and $\omega_{Y}$ with $k\omega_{X}$ and $k\omega_{Y}$. 
Since $L$ is a $\varphi$-ample line bundle on $X$,
we can take a  smooth Hermitian metric $g$ on $L$
such that $\sqrt{-1}\Theta_{g}+\varphi^{*} \omega_{Y} \geq c \omega_{X}$
holds for some constant $1 \gg c>0$ by replacing  $\omega_{X}$ and $\omega_{Y}$ with $k\omega_{X}$ and $k\omega_{Y}$.

Let $Z$ be the $r$-times fiber product  of $\varphi \colon X \to Y$
with the $i$-th projection $\pr_{i}\colon Z\to X$ and the natural morphism $\psi \colon Z \to Y$:
$$
\xymatrix{
    Z:=X \times_{Y}  \cdots \times_{Y}X   \ar[drr]_{\psi} \ar[rr]^{\qquad \qquad  \pr_{j}} \ar[d]_{\pr_{i}} && X \ar[d]^{\varphi} \\
   X       \ar[rr]_{\varphi}  &&   Y.}
$$
Set
$$
L_{r}:=\sum_{i=1}^{r} \pr_{i}^{*} L \text{ and }
L':=L_{r} \otimes  (\psi^{*}\det \varphi_{*} (L))^{*}.
$$
To apply Proposition \ref{prop-direct} to
\begin{align*}
\psi_{0}= \psi|_{Z_{0}}\colon Z_{0} := \psi^{-1}(Y_{0})\to Y_{0}
\text{ equipped with } L' |_{Z_{0}} \text{ and }  \theta:=\eta_{1}
\end{align*}
we  examine the non-nef locus of $-K_{Z/Y}$ and metrics on $L'$.

By Conditions (a), (b) of  Proposition \ref{prop-direct}, we obtain singular Hermitian metrics $\{g_{\delta}\}_{\delta>0}$
on $-K_{X/Y}|_{X_{0}}=-K_{X_{0}/Y_{0}}$ such that
$\sqrt{-1}\Theta_{g_{\delta}} \geq -\delta \omega _{X}$ holds on $X_{0}$
and the upper-level set $\{x \in X_{0}\,|\, \nu(g_{\delta}, x)>0 \}$ of Lelong numbers is not dominant over $Y_{0}$.
Since $\psi\colon Z \to Y$ is a smooth fibration over $Y_{0}$,
we have
$$
K_{Z_{0}}= \sum_{i=1}^{r} \pr_{i}^{*} K_{X_{0}} \text{ on } Z_{0}.
$$
Hence, we obtain the metric $G_{\delta}:=\sum_{i=1}^{r} \pr_{i}^{*} g_{\delta}$ on $-K_{Z_{0}/Y_{0}}$.
By construction, the upper-level set $\{x \in Z_{0}\,|\, \nu(G_{\delta}, x)>0 \}$  of Lelong numbers 
is not dominant over $Y_{0}$ and the curvature current $\sqrt{-1}\Theta_{g_{\delta}}$ satisfies that
$$
\sqrt{-1}\Theta_{g_{\delta}} \geq  -\delta \sum_{i=1}^{r} \pr_{i}^{*} \omega _{X}
\text{ on } Z_{0}.
$$
This indicates that $-K_{Z_{0}/Y_{0}}$ satisfies Condition (a) of Proposition \ref{prop-direct}
for the K\"ahler form  $\sum_{i=1}^{r} \pr_{i}^{*} \omega _{X}$ on $Z_{0}$.
Consider the smooth Hermitian metric
$$
G:=(\sum_{i=1}^{r} \pr_{i}^{*}g) \cdot (\psi^{*}g_{1})^{-1}
\text{ on } L'=(\sum_{i=1}^{r} \pr_{i}^{*} L) \otimes  (\psi^{*}\det \varphi_{*} (L))^{*}.
$$
Recall that $g_{1}$ is the smooth Hermitian metric on $\det \varphi_{*} L  |_{Y_{0}}$ in Condition (d).
Then, we obtain that
$$
\sqrt{-1}\Theta_{G}(L') +
\sum_{i=1}^{r} \pr_{i}^{*} \varphi^{*}\big (\omega_{Y} + \frac{1}{r}\eta_{1} \big) \geq \sum_{i=1}^{r} \pr_{i}^{*} c\,\omega_{X}
\text{ on } Z_{0}.
$$
Here we used $\psi^{*}=\pr_{i}^{*} \circ  \varphi^{*}$ for any $1 \leq i \leq r$.
Since $\omega_{Y} + (1/r)\eta_{1}$ is a K\"ahler form on $Y_{0}$,
the line bundle $L'|_{Z_{0}}$ satisfies Condition (b) of Proposition \ref{prop-direct}.
On the other hand, there exists the non-zero natural morphism
$$
\det \varphi_{*}(L) \to (\varphi_{*}(L))^{\otimes r} \cong \psi_{*} (L_{r}) \text{ on } Y_{0},
$$
which shows that $h^{0}(Z_{0}, L') \not =0$ by the definition of $L'$.
In particular, the line bundle $L'|_{Z_{0}}$ satisfies Condition (c) of Proposition \ref{prop-direct} for $\theta=0$ (and $\delta'=0$).
The above arguments enable us to apply Proposition \ref{prop-direct},
and then we obtain singular Hermitian metrics $\{H_{\e}\}_{\e>0}$ on $\psi_{*}(L') |_{Y_{0}}$
such that $\sqrt{-1}\Theta_{H_{\e}} \geq  -\e \omega_{Y} \otimes \id$ on $Y_{0}$.

We finally prove the desired conclusion
using the metrics on $L'$ induced by $H_{\e}$.
Let us consider the natural morphism
\begin{align}\label{eq-dia}
\psi^{*}\psi_{*}(L') \to L'  \text{ on } Z_{0},
\end{align}
which is generically surjective by $h^{0}(Z_{0}, L') \not =0$.
Let  $G_{\e}$ be the metric on $L' |_{Z_{0}}$ induced
by $\psi^{*}H_{\e}$ and the above morphism.
The diagonal subset $\Delta$ of the fiber product $Z_{0}$ is identified with $X_{0}$.
Note that $L'|_{\Delta} \cong  rL \otimes \big(  \varphi^{*}\det \varphi_{*}L\big)^{*}$ holds
under this identification and \eqref{eq-dia} is not identically zero on $\Delta$.
By construction, the metric $G_{\e}|_{\Delta}$ on
$L'|_{\Delta} \cong rL \otimes ( \varphi^{*}\det \varphi_{*}L)^{*}$
 is well-defined
(i.e.,\,$G_{\e}|_{\Delta} \not \equiv \infty$) and
$$
\sqrt{-1}\Theta_{G_{\e}}|_{\Delta} \geq -\e \psi^{*}\omega_{Y}|_{\Delta} \geq -\e \omega_{X}
\text{ holds on }\Delta \cong X_{0}.
$$
Note that the well-definedness follows since $G_{\e}$ is constructed by the pull-back $\psi^{*}H_{\e}$.
This curvature condition can be extended to $X$ by $\codim (X\setminus X_{0}) \geq 2$.
Here, we used that $\varphi\colon  X\to Y$ has equi-dimensional fibers.
\end{proof}

We finish the proof of Theorem \ref{theo-flat}.
Let $p$ be an integer with $p/r \in \mathbb{Z}_{+}$.
By Claim \ref{claim2} and Condition (d),
there exist singular Hermitian metrics $\{g_{\delta'}\}_{\delta'>0}$ on $L |_{X_{0}}$
such that
$$
\sqrt{-1}\Theta_{g_{\delta'}}(L) \geq \varphi^{*}
\Big( \frac{1}{pr_{p}} \eta_{p} \Big) - \delta' \omega_{X} \text{ on } X_{0}.
$$
Let us apply Proposition \ref{prop-direct} to $\theta:=(1/pr_{p}) \eta_{p}$.
Then, since $\eta_{p}$ is the curvature of $\det \varphi_{*}(pL)$,
we see that
\begin{align}\label{eq-weak}
\varphi_{*}L  \otimes \Big( \frac{1}{pr_{p}} \det \varphi_{*}(pL)\Big)^{*}
\text{ is weakly positively curved}
\end{align}
on $Y_{0}$ (with respect to $\omega_{Y}$);
hence it is weakly positively curved on $Y$ since $\omega_{Y}$ is defined on $Y$.
The determinant sheaf
$$
\det \varphi_{*} L   \otimes \Big( \frac{r}{pr_{p}} \det \varphi_{*}(pL)\Big)^{*}
=\det \varphi_{*} L  - \frac{r}{pr_{p}} \det \varphi_{*}(pL)
$$
is also weakly positively curved on $Y$.
Here, we use the additive notation on the left-hand side, which is justified on $Y_{0}$
(see Remark \ref{rem-flat}).
This implies that
\begin{align*}
(\det \mathcal{V}_{p})^{*}&=
-\det \varphi_{*}(pL) + \frac{r_{p} p}{r} \det \varphi_{*} L \\
&\geq_{\rm{w}} -\det \varphi_{*}(pL) + \frac{r_{p} p}{r}\cdot \frac{r}{pr_{p}} \det \varphi_{*}(pL) =0
\text{ on } Y_{0}
\end{align*}
is weakly positively curved,
where the notation $\geq_{\rm{w}}  $ denotes the difference is weakly positively curved.
On the other hand, since $(p/r)L$ is sufficiently $\varphi$-ample,
the natural morphism
$$
\mathcal{W}_{p}:=\Sym^{p}(\varphi_{*} L ) \otimes
\Big( \frac{p}{r}\det \varphi_{*} L \Big)^{*}
\to
\varphi_{*}(pL)\otimes \Big( \frac{p}{r}\det \varphi_{*} L \Big)^{*}=\mathcal{V}_{p}
$$
is generically surjective for $p \gg1$.
The sheaf $\mathcal{W}_{p}$ can be written as the $p$-th symmetric tensor of
\eqref{eq-weak} of $p=1$;
therefore $\mathcal{W}_{p}$ is weakly positively curved.
By the above morphism, we see that $\mathcal{V}_{p}$
is also weakly positively curved.
\end{proof}

\subsection{Hermitian metrics on orbifold vector bundles}\label{subsec2-4}
In this subsection,  following the discussions in \cite{MM, Wu23},
we review some basic facts on orbifold structures.

Let $(X, \omega)$ be a compact K\"ahler space with quotient singularities.
The space $X$ can be regarded as a K\"ahler orbifold
(see \cite[Section 5.4]{MM}, \cite[Definition 1]{ Wu23} for the precise definition);
in particular, there exists an open cover $\mathcal{U}:=\{U_\alpha\}_{\alpha \in A}$ of $X$
and a Galois quasi-\'etale cover $\pi_\alpha\colon  \tilde{U}_\alpha \to \tilde{U}_\alpha/G_\alpha \cong U_\alpha$,
where $\tilde{U}_\alpha$ is a smooth variety and $G_\alpha$ is a finite group.
Recall that quasi-\'etale covers are defined as finite surjective morphisms that are \'etale in codimension $1$.
The smooth variety $\tilde{U}_\alpha$  is called a \textit{local smooth ramified cover}.

A reflexive  sheaf $\cF$ on $X$ is called an {\textit{orbifold vector bundle}}
if $\{(\pi_\alpha^* \cF)^{**}\}_\alpha$ is locally free for any $\alpha \in A$.
The complex orbifold, denoted by $F$, is determined
by the quotients $\{(\pi_\alpha^* \cF)^{**}/G_\alpha\}_\alpha$
of the total space (as a vector bundle) of $(\pi_\alpha^* \cF)^{**}$,
which has the orbifold morphism $F \to X$.
Furthermore, the projective space orbifold bundle $p \colon  \PP(F) \to X$ and
the tautological orbifold line bundle $\OX_{\PP(F)}(1)$ can be also defined.
A singular Hermitian metric $h$ on $\cF$ in the sense of Subsection \ref{subsec2-1}
determines the {\textit{orbifold metric}},
that is, the family $\{h_{\alpha}\}_{\alpha}$
of a (possibly singular) metric $h_{\alpha}$ (defined by the pullback of $h$) on $(\pi_\alpha^* \cF)^{**}$
compatible with the orbifold structure.
The orbifold metric is said to be {\textit{smooth}} if $h_{\alpha}$ is smooth for any $\alpha$.
Then, we can define the notation of positivity for $\cF$ as in the case of locally free sheaves,
and generalize \cite[Theorem 1.18]{DPS94} to orbifold vector bundles.
We emphasize that orbifold vector bundles themselves are sheaves,
but all the metrics and positivity are calculated on local smooth ramified covers.

\begin{defi}
\label{orb-metric}
Let $\omega$ be a K\"ahler form on a K\"ahler orbifold $X$.
An orbifold vector bundle $\cF$ is said to be {\textit{pseudo-effective in the strong sense}}
if there exists a family $\{h_\e\}_{\e>0}$ of orbifold metrics on $\OX_{\PP(F)}(1)$ such that
$$
\sqrt{-1}\Theta_{h_\e}(\OX_{\PP(F)}(1)) +  \e p^* \omega \geq 0
\text{ holds  in the sense of currents }
$$
and the polar set $\{h_\e=+\infty\}$ is not dominant over $X$,
where the above inequality means that it holds on each local smooth ramified cover of $\PP(F)$.
The orbifold vector bundle $\cF$ is said to be {\textit{nef}}
if  $h_\e$ in the above condition can be chosen as a smooth orbifold metric.

The orbifold vector bundle $\cF$ is said to be \textit{{numerically flat}}  (resp.\,{\textit{Hermitian flat}})
if both $\cF$, $\cF^*$ are nef (resp.\,$\cF$ admits an orbifold metric whose curvature vanishes
over each local smooth ramified cover).
\end{defi}

\begin{theo}[{\cite[Corollary 2, Theorem D]{Wu23}}] \label{thm-flat}
Let $\cF$ be an orbifold vector bundle on a compact K\"ahler orbifold $X$.

$(1)$  If $\cF$ is weakly positively curved as a sheaf on $X$ $($see Definition \ref{def-psef}$)$
and satisfies that $c_{1}(\cF)=0$, then $\cF$ is a numerically flat orbifold vector bundle.

$(2)$ The orbifold vector bundle $\cF$ is numerically flat if and only if there exists a filtration of orbifold vector bundles
\begin{align}\label{fil-herm}
0=:\mathcal{F}_0 \subset \mathcal{F}_{1} \subset \dots \subset \mathcal{F}_{m-1} \subset
\mathcal{F}_{m}:=\cF
\end{align}
such that each quotient  $\mathcal{F}_k/\mathcal{F}_{k-1}$ is a Hermitian flat orbifold vector bundle.
\end{theo}
\begin{proof}
Let $\{h_{\e}\}_{\e>0}$ be singular Hermitian metrics  on $\cF$ satisfying
the definition of weakly positively curved sheaves (see Definition \ref{def-psef}).
We can easily see that the  metrics on $\OX_{\PP(F)}(1)$ induced by $\{h_{\e}\}_{\e>0}$
satisfies the definition of pseudo-effective orbifold vector bundles in the strong sense.
By \cite[Lemma 1]{Wu23}, a K\"ahler form on $X$ defines an orbifold K\"ahler form
modulo $ \ddbar$-exact forms with continuous potential.
In particular, we may assume that $\omega$ in Definition \ref{orb-metric}
is an orbifold K\"ahler form.
Thus, Conclusion (1) follows from \cite[Corollary 2]{Wu23}.
Conclusion (2) is a direct consequence of  \cite[Theorem D]{Wu23}.
\end{proof}

\begin{cor} \label{cor-flat}
Let $\cF$ be a numerically flat orbifold vector bundle on a compact K\"ahler orbifold $X$.
Assume that $\pi_{1}(X_{\reg})=\{\id\}$ and $H^1(X_\reg, \mathcal{O}_{X})=0$.
Then, the sheaf $\cF$ is a trivial vector bundle on $X$.
\end{cor}
\begin{proof}
Take the filtration as in Theorem \ref{thm-flat}.
The quotient sheaf $\mathcal{F}_k/\mathcal{F}_{k-1}$ is a Hermitian flat vector bundle on $X_{\reg}$ (but not necessarily on $X$);
therefore, it is induced by a $\rm{GL}$-representation of the (topological) fundamental group $\pi_{1}(X_{\reg})$,
which indicates that $\mathcal{F}_k/\mathcal{F}_{k-1}$ is a trivial vector bundle.
The extension class of $0 \to \mathcal{F}_k \to \cF \to \mathcal{F}_k/\mathcal{F}_{k-1} \to 0$
(which is an exact sequence of vector bundles on $X_{\reg}$) is trivial by $H^1(X_\reg, \mathcal{O}_{X})=0$.
This indicates that $\cF$ is trivial on $X_{\reg}$, and thus it is trivial on $X$ by reflexivity.
\end{proof}

\section{$\mathbb{Q}$-Conic Bundles and the Minimal Model Program}\label{Sec3}

$\mathbb{Q}$-conic bundles naturally appear as an outcome of the MMP in our situation
(see Subsection \ref{subsec3-3} for details).
For this reason, we respectively study $\mathbb{Q}$-conic bundles and conic bundles
in Subsections  \ref{subsec3-2} and \ref{subsec3-1}.
Furthermore, we clarify what the nefness of anti-canonical bundles brings to the geometry of $\mathbb{Q}$-conic bundles.

\subsection{$\mathbb{Q}$-conic bundles}\label{subsec3-2}
In this subsection, following \cite{MP08a, MP08b, Pro07},
we summarize the basic properties of $\QQ$-conic bundles.
We first review the definition of $\QQ$-conic bundles.

\begin{defi}\label{def-Qconic}
(1)
Let $X$  and $S$ be normal analytic varieties.
A fibration $\varphi\colon X \to S$ is called \textit{a $\QQ$-conic bundle} if it satisfies following conditions:
\begin{itemize}
    \item $X$ has terminal singularities;
    \item $\varphi\colon X \to S$ is equi-dimensional and of relative dimension $1$;
    \item $-K_X$ is $\varphi$-ample.
\end{itemize}
Throughout this paper, except for Subsection \ref{subsec3-1},
we consider only a $\QQ$-conic bundle $\varphi\colon X \to S$ with $\dim X=3$.

(2)
\textit{The discriminant divisor $\Delta$} of a $\QQ$-conic bundle $\varphi\colon  X \to S$  is defined by
the union of divisorial components of the non-smooth locus
$
\{s \in S \, |\, \varphi \mathrm{\; is \; not \; a\;  smooth\;  fibration \; at \;s}\}.
$

(3)
A $\QQ$-conic bundle $\varphi\colon  X \to S$ is said to be \textit{toroidal} at $s \in S$
with respect to $\mu_m:=\mathbb{Z}/m\mathbb{Z}$
if  $X$ is isomorphic to the quotient  $\PP^1 \times \CC^2/\mu_m$ over a neighborhood of $s$.
Here the $\mu_m$-action is defined by
$$
(t; z_1, z_2) \to (\e^b t; \e z_1, \e^{-1} z_2),
$$
where $b$ is an integer with $\mathrm{gcd} (m,b)=1$ and $\e$ is the primitive $m$-th root of unity.
In this case, the singularities of $X$ are cyclic quotient singularities of types $(1/m)(b,1,-1)$ and $(1/m)(-b,1,-1)$;
furthermore, the singularities of the base $S \cong \CC^2/\mu_m$ are the cyclic quotient of type $A_{m-1}$.
\end{defi}

A $\QQ$-conic bundle $\varphi\colon X \to S$  with $\dim X=3$ can be explicitly described locally over $S$.
The following corollary is a direct consequence of the classification \cite[Corollary 10.85]{Pro18}.

\begin{lemm} [{\cite[Corollary 10.85]{Pro18}}]\label{MP}
Let $\varphi\colon  X \to S$ be a $\QQ$-conic bundle with $\dim X=3$ and  $\Delta \subset S$
be the discriminant divisor.
Then $s \not \in \Delta$ if and only if $\varphi\colon  X \to S$ is toroidal at $s$.
\end{lemm}

\subsection{Conic bundles} \label{subsec3-1}

In this subsection, we review conic bundles as presented in \cite{Sar82} and extend certain properties to the non-projective case
(see Propositions \ref{basic-conic-bundle} and \ref{push-forward-formula}).
We derive Proposition \ref{push-forward-formula} from Proposition \ref{basic-conic-bundle}.
Proposition \ref{basic-conic-bundle} is proved by the same argument as in \cite{Sar82} even in the non-projective case;
hence we omit the detailed proof.

\begin{defi} \label{conic-bundle}
A $\mathbb{Q}$-conic bundle $\varphi\colon  X \to S$ is called a \textit{conic bundle} if
$X$ and $S$ are smooth.
\end{defi}

\begin{prop} \label{basic-conic-bundle}
Let $\varphi\colon  X \to S$ be a conic bundle. Then we have$:$
\begin{enumerate}
\item
$-K_X$ is $\varphi$-very ample;
$E := \varphi_* (-K_X)$ is a locally free sheaf of rank $3$;
$E = \varphi_* (-K_X)$ defines an embedding of $\varphi\colon  X \to S$ into $p\colon \mathbb{P}(E) \to S$$:$
\begin{equation*}
\xymatrix{
X \ar@{^{(}->}@<0.3ex>[rr] \ar[rd]_{\varphi}&&  \PP(E)\ar[ld]^{p} \\
&S.   &
}
\end{equation*}
Furthermore, the scheme-theoretic fiber $X_s$ at a point $s \in S$ is a $($possibly reducible or non-reduced$)$ conic
on the projective plane $\PP(E_s)(\cong \mathbb{P}^{2})$.

\item  $X \subset \PP(E )$ can be written as the zero locus of a section
$$\sigma \in H^0(\PP(E), \OX_{\PP(E)}(2) \otimes p^*(-\det E-K_S)).$$

\item We identify $\PP(E )$ with $\PP^2 \times U$ over a small open subset $U \subset S$ with a coordinate $z$.
Then, the embedding $X \subset \PP(E ) = \PP^2 \times U$ over $U$
can be written as
$$
X = \{([x_{0} : x_{1} : x_{2}], z) \in \PP^2 \times U \,|\,\sum_{0 \leq i,j \leq 2} a_{i,j}(z) x_i x_j = 0\},
$$
where $a_{i,j} \in \OX_{S}(U)$.
Furthermore, the discriminant divisor $\Delta$ coincides  with
the non-smooth locus
$
\{s \in S \, |\, \varphi \mathrm{\; is \; not \; smooth \; at \;s}\}
$
and can be described as
$$
\Delta = \{z \in S  \,|\, \det [a_{i,j}(z)]_{i,j=0}^{2}  = 0\}.
$$

\item For a given point $s \in U$, by changing a coordinate $z$ of  $ U$ and a basis of $E$,
we can assume that $a_{i,j}=0$ for $i \not =j$.
Furthermore, we can assume that
\begin{itemize}
\item $a_{i,i} \in \OX_{Y}^*(U)$ for any $0\leq i \leq 2$ when $X_{s}$ is smooth;
\item $a_{i,i} \in \OX_{Y}^*(U)$ for any $1\leq i \leq 2$  and ${\rm{mult}}_{s}(a_{0,0})=1$
when $X_{s}$ is reduced and reducible.
\end{itemize}

\item The discriminant divisor $\Delta$  is normal crossing in codimension $2$ in $S$.
Furthermore, the fiber $X_s$ of a general point $s \in \Delta$ is reduced and reducible.

\item $c_1(\Delta)=-c_1(\det E)-3c_1(K_S)$ holds.
\end{enumerate}
\end{prop}
\begin{proof}
Points (1) and (2) are due to \cite[1.5]{Sar82}.
Point (3) is due to \cite[Proposition 1.7]{Sar82}.
Point (4) is due to \cite[Proposition 1.8, Point 5]{Sar82}.
Point (6) is due to \cite[Definition 1.6]{Sar82}.
The divisor $\Delta$ is reduced by \cite[Corollary 1.9]{Sar82}.
The first statement of Point (5) follows
from the proof of \cite[Proposition 1.8, Point 5]{Sar82} (see also \cite[3.3.3]{Pro18}).
The second statement of Point (5) follows from \cite[Proposition 1.8, Point 4]{Sar82}.
\end{proof}

\begin{prop} \label{push-forward-formula}
 Let $\varphi\colon  X \to S$ be a conic bundle. Then, we have
$$\varphi_*(c_1(K_X)^2)=-4c_1(K_S)-c_1(\Delta).$$
\end{prop}
\begin{proof}
The proposition is proved in the projective case (see \cite[4.11]{Mi83}).
We extend the discussion of \cite[4.11]{Mi83} to the non-projective case.
By Proposition \ref{basic-conic-bundle} (6),
the conclusion is equivalent to
$$
\varphi_*(c_1(K_X)^2)=\frac{4}{3}c_1(\det E)+\frac{1}{3}c_1(\Delta).
$$
We fix a smooth Hermitian metric $h$ on $E= \varphi_* (-K_X)$
and use the same notation $h$ to denote the induced metric on $\OX_{{\mathbb{P}(E)}}(1)$.
Then, by the adjunction formula and Proposition \ref{basic-conic-bundle} (2),
the conclusion is equivalent to the following formula:
\begin{align}\label{eq-goal}
p_*(c_1(\OX_{\mathbb{P}(E)}(1),h)^2 \wedge [X])=\frac{4}{3}c_1(\det E, \det h)+\frac{1}{3}[\Delta],
\end{align}
where $[X], [\Delta]$ are the integration currents and
$c_1(\OX_{\mathbb{P}(E)}(1),h), c_1(\det E, \det h)$ are the Chern curvatures divided by $2\pi$.

We first prove the desired formula on $S \setminus \Delta$.
For this purpose, we summarize some formulas for the curvatures of vector bundles
(e.g.,\,see \cite[Section 15.C, Chap.\,V] {agbook}).
Set  $n=\dim S$ and $r:=\rank E (=3)$.
For a given $s \in S\setminus \Delta$, we take a local frame $(e_\lambda )_{\lambda=1}^{r}$ of $E$
giving an orthonormal basis of $E_{s}$ at $s \in S$,
and then write the Chern curvature of $E$ as
\begin{align}\label{eq1}
\Theta_{h}(E) = \sum_{1 \leq j,k \leq n, 1 \leq \lambda, \mu \leq r} c_{jk \lambda \mu} dz^j \wedge d\overline{z}^k \otimes e^*_\lambda \otimes  e_\mu,
\end{align}
where $(z_j)_{j=1}^{n}$ is a local coordinate of $S$.
Let $[x] \in \PP (E_s)$ be the point represented
by a vector $\sum_{\lambda=1}^{r} x_\lambda e^*_\lambda \in E^*_s$
with $\sum_{\lambda=1}^{r} |x_\lambda|^{2}=1$.
Then, the curvature of $\OX_{\PP(E)}(1)$ at $[x]$ can be written as
\begin{align}\label{eq2}
\Theta_{h}(\OX_{\PP(E)}(1))_{[x]} = \sum_{1 \leq j,k \leq n, 1 \leq \lambda, \mu \leq r} c_{jk \lambda \mu} x_\lambda \overline{x}_\mu dz^j \wedge d\overline{z}^k +\sum_{1 \leq \lambda \leq r-1} d \xi_\lambda \wedge d \overline{\xi}_\lambda,
\end{align}
where $(\xi_\lambda)_{\lambda=1}^{r-1}$ is the  coordinate of $\PP (E)$
induced by unitary coordinates on the hyperplane $(\mathbb{C} x)^{\perp} \subset E^*_s$.
We identify $\PP(E)  $ with $\PP^{r-1} \times U$ over an open neighborhood $U \subset S$ of $s$,
and regard the Fubini-Study form $\Omega$  on $\PP^{r-1}$ as the $(1,1)$-form on $\PP(E)$.
Then,  we can easily check that
$$
\frac{\sqrt{-1}}{2 \pi} \Theta_{h}(\OX_{\PP(E)}(1))_{[x]} = \Omega-
\frac{\sqrt{-1}}{2 \pi}\frac{\langle p^*  \Theta_{h^*}(E^*)x, x \rangle }{|x|^2}
\quad \text{ for any } [x] \in p^{-1}(s).
$$
The pushforward of smooth forms can be described as a fiber integration near $s \in S \setminus \Delta$
(e.g.,\,see \cite[Theorem 1.14, Proposition 2.15, Chap.\,II]{GPR94}).
Hence, we obtain
\begin{align*}
p_*(c_1(\OX_{\mathbb{P}(E)}(1),h)^2 \wedge [X])_{s}&=
\int_{X_s} c_1(\OX_{\PP(E)}(1),h)^2 \\
&= 
-2 \cdot   \frac{\sqrt{-1}}{2 \pi} \int_{X_s} \Omega \wedge \frac{\langle p^* \Theta_{h^*}(E^*)x, x \rangle }{|x|^2},
\end{align*}
where $X_s$ is the fiber of $\varphi\colon  X \to S$ at $s \in S$.

We consider the special case of $X_s=\{x_0^2+x_1^2+x_2^2=0\}.$
We can easily see that
$$
\int_{X_s} \frac{x_\lambda \overline{x}_\mu} {|x|^2} \, \Omega =\frac{2}{3} \cdot \delta_{\lambda \mu},
$$
where $\delta_{\lambda \mu}$ is the Kronecker  delta.
Hence, by the formula \eqref{eq1}, we obtain
\begin{align*}
\int_{X_s} c_1(\OX_{\mathbb{P}(E)}(1),h)^2&= \frac{4}{3} \frac{\sqrt{-1}}{2 \pi} \sum_\lambda c_{jk \lambda \lambda} dz^j \wedge d \overline{z}^k \\ &=\frac{4}{3}c_1(\det E, \det h).
\end{align*}
The general case can be reduced to the special case.
Indeed,  by Proposition \ref{basic-conic-bundle} (4), we may assume that
$X_s=\{C_0 x_0^2+ C_1 x_1^2+ C_2x_2^2=0\}$,
where $C_i $ is a non-zero constant.
Hence, we see that $X_s$ is cohomologeous to $\{x_0^2+x_1^2+x_2^2=0\}$ in $\PP(E_s)$.
By the Stokes formula, we can conclude that
\begin{align*}
\int_{X_s} c_1(\OX_{\mathbb{P}(E)}(1),h)^2&= \int_{\{x_0^2+x_1^2+x_2^2=0\}} c_1(\OX_{\mathbb{P}(E)}(1),h)^2\\ &=\frac{4}{3}c_1(\det E, \det h).
\end{align*}
The Fubini theorem shows that
$$
\int_{S} \varphi_{*}(c_1(\OX_{\mathbb{P}(E)}(1),h)^2) \wedge \alpha =\int_{S \setminus \Delta } \alpha  \wedge \int_{X_s} c_1(\OX_{\mathbb{P}(E)}(1),h)^2  $$
for any smooth form $\alpha$ with $\Supp \alpha \Subset S \setminus \Delta$,
which implies that
$$
p_*(c_1(\OX_{\mathbb{P}(E)}(1),h)^2 \wedge [X])=\frac{4}{3}c_1(\det E, \det h)
\text{ on } S \setminus \Delta.
$$

We finally prove the desired formula \eqref{eq-goal} on $S$.
The pushforward $p_*(c_1(\OX_{\mathbb{P}(E)}(1),h)^2 \wedge [X])$ is a normal current;
hence, by applying the support theorem (see Lemma \ref{support-lemma}),
we can find $c_i \in \RR$ such that
$$
\mathbbm{1}_{\Delta} \, p_*(c_1(\OX_{\mathbb{P}(E)}(1),h)^2 \wedge [X])=\sum_i c_i [\Delta_i],
$$
where $\{\Delta_i\}_{i \in I}$ are the irreducible components of $\Delta$.
We show that $c_i=1/3$ for any $i$ by regarding $c_i$
as the generic Lelong number along $\Delta_{i}$.
Note that $c_i$ is independent of the choice of metric $h$ on $E$.
Indeed, let $h'=h e^{-\psi}$ be another smooth Hermitian metric on $\OX_{\mathbb{P}(E)}(1)$ with some $\psi \in C^{\infty}(\PP(E))$.
Then, we have
$$
p_*(c_1(\OX_{\mathbb{P}(E)}(1),h')^2 \wedge [X])-p_*(c_1(\OX_{\mathbb{P}(E)}(1),h)^2 \wedge [X])=\frac{\sqrt{-1}}{2 \pi} \partial \overline{\partial} F,
$$
where $F$ is a function defined by
$$
F(s)=\int_{X_s} \psi  \cdot (c_1(\OX_{\mathbb{P}(E)}(1),h')+c_1(\OX_{\mathbb{P}(E)}(1),h))
\text{ for } s \in S.
$$
To check that the Lelong number is independent of the choice of $h$,
it suffices to show that $F$ is continuous
over any small $1$-dimensional disc passing through a general point of $\Delta$,
but this follows from the theory of cycle spaces (e.g.,\,see \cite[Corollaire 1]{Bar78}).

To finish the proof, for a general point $s \in \Delta_i$,
we  construct a smooth Hermitian metric $h$ on $E$ (e.g.,\,by a partition of unity)
such that  $h$ is flat on a neighborhood of $s$.
Then, by $c_1(\OX_{\mathbb{P}(E)}(1),h)=\Omega$,
it suffices to show that
$$
p_*(\Omega^2 \wedge [X])=\frac{1}{3} [\Delta_i]
$$
on a neighborhood $U$ of $s$.
In the following, we locally approximate $X \to S$ with projective conic bundles $\{X_{N} \to S\}_{N=1}^{\infty}$,
and prove the desired equality by using \cite[4.11]{Mi83}.
(The proposition itself does not seem to be directly derivable from \cite[4.11]{Mi83},
even when considering the following approximate argument.)
Let $z_0 x_0^2+f_1 x_1^2 +f_2 x^2_2$ be
a local defining function of $X \subset \PP(E)$,
where $f_1, f_2 \in \OX^*(U)$.
Let us regard $U \subset S$ as an open subset in $\PP^n$.
Then, we can find $g_N \in H^0(\PP^n \times \PP^{2}, \OX(N)  \boxtimes \OX(2))$
such that
$$
[X \cap \PP(E)]=\lim_{N \to \infty}[X_{N} \cap \PP(E)] \text{ over }U\text{, where }X_{N}:=\{g_{N}=0\}
$$
by using the polynomial  approximation of $f_1$, $f_2$ (e.g.,\,we can use the Taylor expansion).
By the Bertini theorem,
a general member in $H^0(\PP^n \times \PP^{2}, \OX(N)  \boxtimes \OX(2))$
determines a conic bundle over $\PP^n$.
Thus, by replacing $g_{N}$ with the general member, we may assume that
$X_{N}=\{g_{N}=0\} \to \PP^n$ is a conic bundle with discriminant divisor $\Delta_{N}$.
Since $X_{N} \to \PP^n$ is a projective conic bundle,
we have
$$
p_*(\Omega^2 \wedge [X_{N} \cap \PP(E)])=
\frac{1}{3} [\Delta_{N}] \text{ over }U
$$
by  \cite[4.11]{Mi83}.
Consequently, by taking the limit as $N \to \infty$ in the space of currents, we obtain the desired conclusion.
Indeed, the pushforward $p_*(\Omega^2 \wedge \bullet)$ defines a continuous map
from the space of currents on  $U \times \mathbb{P}^2$ to that of $U$,
which implies that the left-hand side converges to $p_*(\Omega^2 \wedge [X \cap \PP(E)])$.
On the other hand, the right-hand side converges to $[\Delta_i]$ by construction. 
\end{proof}

\subsection{Minimal Model Program}\label{subsec3-3}
In this subsection, we review the MMP for K\"ahler three-folds developed in \cite{HP15a, HP15b, HP16},
and observe what the MMP brings to Conjecture \ref{conj-main} (see Corollary \ref{push-forward-sing}).

\begin{theo} [\cite{HP15b}]\label{3-fold-MMP}
Let $X$ be a $\QQ$-factorial compact K\"ahler space of dimension $3$ with terminal singularities.
Assume that $\dim R(X)=2$, where $R(X)$ is the base of an MRC fibration $X \dashrightarrow R(X)$ of $X$.
Then, we have$:$

$(1)$ $X$ is bimeromorphic to an MFS $($Mori fiber space$)$;
more precisely, there exist
$$
\text{
a bimeromorphic map $\pi\colon  X \dashrightarrow X'$ and an MFS $\varphi\colon  X' \to S$ such that
}
$$
\begin{itemize}
\item[(a)] $X \dashrightarrow X'$ is obtained from the composition of divisorial contractions and flips;
\item[(b)] $X'$ is a $\QQ$-factorial compact K\"ahler space  with terminal singularities;
\item[(c)] $S$ is a $\QQ$-factorial compact K\"ahler space of dimension $2$ with klt singularities;
\item[(d)] $S$ is non-uniruled and $K_{S}$ is pseudo-effective;
\item[(e)] $-K_{X'}$ is $\varphi$-ample and the relative Picard number $\rho(X'/S) $ is $1$;
\item[(f)] $\varphi\colon  X' \to S$ is equi-dimensional and of relative dimension $1$.
\end{itemize}

$(2)$ %
The outcome $X \dashrightarrow X'  \to S$ of the MMP
factors through the Albanese map $\alpha\colon X \to A(X)$, that is,
there exists the morphism $\beta\colon S \to A(X)$ with the diagram$:$
$$
\xymatrix{
    X \ar@{.>}[r]^{\pi} \ar[d]_{\alpha} & X' \ar[d]^{\varphi} \\
   A(X)        & \ar[l]_{\beta}  S.}
$$
\end{theo}
\begin{proof}
By running the MMP in \cite{HP15b} for the initial variety $X$,
we can find a bimeromorphic map $\pi\colon  X \dashrightarrow X'$
such that either $K_{X'}$ is nef
or there exists an MFS $\varphi\colon  X' \to S$.

We can easily exclude the case where $K_{X'}$ is nef.
Indeed, by \cite[Corollary 1.2]{Bru06}, we see that $Y$  is non-uniruled if and only if $K_{Y}$ is pseudo-effective,
where $Y$ is a compact K\"ahler space of dimension $\leq 3$ with terminal singularities.
Hence, by noting that $X'$ has terminal singularities,
if $K_{X'}$ is nef, the variety $X'$ is non-uniruled, which contradicts $\dim R(X)=2$.
Hence, except for (c) and (d), the other properties follow from \cite{HP15b}.

The outcome $X \dashrightarrow X'  \to S$ gives one of MRC fibrations of $X$,
which implies that $S$ is a non-uniruled surface by $\dim R(X)=2$.
The non-uniruledness shows that $K_{S}$ is pseudo-effective.
Indeed, for a minimal resolution $\pi\colon   \bar S \to S$ of $S$,
we have $\pi^{*} K_{S} =K_{\bar S} + E$ for some effective exceptional $\mathbb{Q}$-divisor $E$.
Since $\bar S$ is non-uniruled, we see that $K_{\bar S}$ is pseudo-effective;
hence so is $K_{S}$.

As in the case of the projective case,
all the steps of the MMP (i.e.,\,divisorial contractions, flips, MFSs)
are obtained from contractions of rational curves.
Thus, since the torus $A(X)$ has no rational curve,
the outcome $X \dashrightarrow X' \to S$ factors through the Albanese map $X \to A(X)$.
\end{proof}

Let us briefly examine the positivity of $-K_{X'}$.
Suppose that we start from $X$ with the nef anti-canonical bundle $-K_{X}$,
and obtain a bimeromorphic map $\pi \colon X \dashrightarrow X'$ in Theorem \ref{3-fold-MMP}.
We might expect that $-K_{X'}$ is still nef ``outside the exceptional locus'' of $\pi\colon  X \dashrightarrow X'$,
but proving this is not straightforward.
In fact, after taking a smooth form $T_{\e} \in c_{1}(-K_{X})$  such that $T_{\e} \geq -\e \omega$,
we can obtain $\pi_{*}T_{\e} \geq -\e \pi_{*}\omega$ on $X'$, where $\omega$ is a K\"ahler form on $X$.
However, it is unclear how the current $\pi_{*}\omega$ relates to a K\"ahler form on $X'$.
For this reason, we prepare Lemma \ref{lemma-cohomology-compare} to compare K\"ahler forms on $X$ to those on $X'$.

\begin{setup}\label{setup}
Before stating Lemma \ref{lemma-cohomology-compare}, we fix the notation.
Assume that $X $ in Theorem \ref{3-fold-MMP} is smooth and  $-K_{X}$ is nef.
The bimeromorphic map $X \dashrightarrow X'$ is decomposed as follows:
\begin{align}\label{eq-bir}
X=:X_0 \dashrightarrow X_1 \dashrightarrow \cdots \dashrightarrow X_{N}:=X',
\end{align}
where each bimeromorphic map $\pi_i\colon  X_i  \dashrightarrow X_{i+1}$ is a divisorial contraction or flip.
Let $\bar X$ be a compact K\"ahler manifold with a bimeromorphic morphism $p_{i}\colon \bar X \to X_{i}$
that resolves the indeterminacy locus of $\pi_{i}$ (when $\pi_{i}$ is a flip).
Depending on whether $\pi_{i}$ is a divisorial contraction or flip,
we obtain the following diagrams:
$$
\xymatrix{
\bar X  \ar[d]_{p_{i}}    \ar@/^10pt/[drr]^{p_{i+1}}   & &  \\
X_{i}  \ar[rr]^{\pi_{i}} && X_{i+1},
}
\xymatrix{
\bar X \ar[dr]^{p_{i}} \ar@/^20pt/[drrr]^{p_{i+1}} \ar@/_30pt/[ddrr]_{r_{i}}& &\\
&X_{i} \ar@{.>}[rr]_{\pi_{i}} \ar[dr]_{q_{i}} && X_{i+1}\ar[dl]^{q_{i+}} \\
&&Z_{i}. &
}
$$
Note that $Z_i$ and $X_{i}$ are K\"ahler spaces by  \cite[Theorem 3.15]{HP15b}.
The varieties $X_i$ are biholomorphic to each other on a non-empty Zariski open set.
This open set, regarded as an open subset of all $X_i$'s,  is called a {\textit{biholomorphic locus}}.
Note that the complement of the biholomorphic locus in $X_{N}=X'$ is of codimension $\geq 2$.

\end{setup}
\begin{lemm} \label{lemma-cohomology-compare}
We consider Setting \ref{setup}.
Then, for any $0 \leq i \leq N$, there exists a K\"ahler form $\omega_{i}$  on $X_{i}$
such that  the Bott-Chern class
$$
\{ p_{0*} (p_{i+1}^* \omega_{i+1} -  p_{i}^* \omega_{i})\}+O(E, K_X)
$$
is represented by a positive current that is smooth on the biholomorphic locus of $X \dashrightarrow X'$,
where  $O(E, K_X)$ is a linear combination of the first Chern classes of $K_X$ and
the exceptional divisors $p_{0*}p_{i}^{*}E_{i}$.
Here $E_{i}$ denotes
the exceptional divisor of $\pi_{i}\colon X_{i} \to X_{i+1}$ $($when it is a divisorial contraction$)$.
In particular, the Bott-Chern cohomology class $
\{ p_{0*} p_{i}^* \omega_{i} -   \omega_{0}\}+O(E, K_X)
$ is represented by a positive current that is smooth on the biholomorphic locus of $X \dashrightarrow X'$.
\end{lemm}
\begin{rem}\label{rem-dimension}
The proof of Lemma \ref{lemma-cohomology-compare} works for K\"ahler spaces of any dimension,
provided that divisorial contractions or flips $\pi_i\colon X_i \dashrightarrow X_{i+1}$ exist.
In the projective case, Mori's cone theorem and Kawamata's base-point-free theorem
contract the extremal ray whose intersection number with the canonical divisor is strictly negative.
This contraction is given by the Stein factorization of the Iitaka fibration of a suitably chosen line bundle.
As a result, the outcomes of MMP are always projective varieties.
It is conjectured that a K\"ahler variant of Mori's cone theorem would similarly contract K\"ahler spaces to K\"ahler spaces, which is verified up to dimension 3 (see \cite{HP16}).
\end{rem}

\begin{proof}
Fix a K\"ahler form $\omega_{N}$ on $X_{N}$.
Assuming $\omega_{i+1}$ has been constructed,
we proceed to inductively construct $\omega_{i}$.

We initially consider the case where $\pi_{i}\colon X_{i} \to X_{i+1}$ is a divisorial contraction
with the exceptional divisor $E_i$.
Since $-E_i$ is $\pi_{i}$-ample,
we can take a smooth form $\theta_{i} \in c_{1}(E_{i})$ such that
$\omega_{i}:=\pi_{i}^* \omega_{i+1}-\e \theta_{i}$ is a K\"ahler form on $X_{i}$
for  $1\gg \e> 0$. Then, we see that
$$
 p_{0*} (p_{i+1}^* \omega_{i+1} -  p_{i}^* \omega_{i})
= p_{0*} p_{i}^* (\pi_i^* \omega_{i+1}-   \omega_{i})=\e p_{0*} p_{i}^{*} \theta_{i}.
$$
The current $p_{0*} p_{i}^{*} \theta_{i}$ represents $c_{1}(p_{0*} p_{i}^{*} E_{i})=O(E)$,
which  finishes the proof.

We now consider the case where $\pi_{i} \colon X_{i} \dashrightarrow X_{i+1}$ is a flip.
Fix a K\"ahler form $\omega_{Z_i}$ (up to a rescaling) on $Z_{i}$
such that $\omega_{i+1} \geq q_{i+}^* \omega_{Z_i}$.
Furthermore, since $-K_{X_i}$ is $q_{i}$-ample,
we can take a smooth form $\eta_{i} \in c_{1}(K_{X_i})$ such that
$\omega_{i}:=q_{i}^* \omega_{Z_{i}}-\e \eta_{i}$ is a K\"ahler form on $X_{i}$ for $1\gg \e> 0$.
Then, we can easily see that
\begin{align*}
p_{0*} (p_{i+1}^* \omega_{i+1} -  p_{i}^* \omega_{i})
&= p_{0*} (p_{i+1}^* \omega_{i+1} -  p_{i}^* (q_{i}^* \omega_{Z_i} -\e \eta_{i}))\\
&= p_{0*} (p_{i+1}^* (\omega_{i+1} -   q_{i+}^* \omega_{Z_i}) +\e p_{i}^*\eta_{i}).
\end{align*}
The current $ \varepsilon p_{0*}  p_{i}^*\eta_{i}$ represents $\varepsilon c_{1} (p_{0*}  p_{i}^*K_{X_i})=O(K_X, E) $
and $p_{0*} (p_{i+1}^* (\omega_{i+1} -   q_{i+}^* \omega_{Z_i}))$ is smooth on the biholomorphic locus,
which finishes the proof.
\end{proof}

Later, we will show that $X \dashrightarrow X'$ described in Theorem \ref{3-fold-MMP} is actually an isomorphism when $-K_X$ is nef.
For this purpose,  we need the following corollary on the intersection number.
Note that when $X_i$ is smooth, some in the following proposition can be simplified
by saying that $-K_{X_i}$ is modified nef.
However, we do not use this terminology in this paper
due to the ambiguity of the definition of modified nefness on singular spaces.

\begin{prop}\label{prop-intersection}
We consider Setting \ref{setup}.
Let  $\omega_{i}$ be a K\"ahler form on $X_{i}$.
Then, we have$:$

$(1)$ There exists a positive current $T_{\e} \in c_{1}(-K_{X_{i}}) + \e \{\omega_{i}\}$
such that $T_{\e}$ is smooth on the biholomorphic locus of $X \dashrightarrow X'$.

$(2)$  For a surface $V \subset X_{i}$,  we have
$$(c_{1}(K_{X_{i}})^2 \cdot \{\omega_{i}\}) \geq 0 \text{ and }
(c_{1}(-K_{X_{i}})  \cdot c_{1}(V) \cdot \{\omega_{i}\}) \geq 0.
$$
\end{prop}
\begin{proof}
By the nefness of $-K_{X}$,
there exists a smooth (semi-)positive form $S_{\e} \in c_{1}(-K_{X})+\e\{\omega_{0}\} $.
By Lemma \ref{lemma-cohomology-compare},
we can find a positive current $P \in \{ p_{0*} p_{i}^* \omega_{i} -   \omega_{0}\}+O(E, K_X)$
such that $P$ is smooth on the biholomorphic locus of $X \dashrightarrow X'$.
The current $p_{i*}p_{0}^{*} (S_{\e} + \e P)$ (defined on $X_i$) is positive
and represents
$$
p_{i*}p_{0}^{*}(c_{1}(-K_{X})
+\e p_{0*} p_{i}^* \{\omega_{i}\}+\e O(E, K_X)).
$$
The above class coincides with $c_{1}(-K_{X_{i}}) + \e  \{\omega_{i}\} + \e O(K_{X_{i}})$
on the biholomorphic locus, which is a Bott-Chern cohomology class on $X_{i}$.
Hence,  by Proposition, \ref{BEG} the positive current $p_{i*}p_{0}^{*} (S_{\e} + \e P)$ actually represents
$c_{1}(-K_{X_{i}})+\e \{\omega_{i}\} + \e O(K_{X_{i}})$.
Then, Conclusion (1) easily follows since the $O(K_{X_{i}})$-part can be absorbed into the K\"ahler class
and $p_{i*}p_{0}^{*} (S_{\e} + \e P)$ is smooth on the biholomorphic locus.

To prove Conclusion (2), we first remark that $c_{1}(V)$ is well-defined since $X_{i}$ is $\mathbb{Q}$-factorial.
Let $Q$ be a $(2,2)$-class defined by  either $c_{1}(V) \cdot \{\omega_{i}\}$
or $c_{1}(-K_{X_{i}}) \cdot \{\omega_{i}\}$.
In any case, noting that $p_{i}^{*}(-K_{X_i})$ and $p_{i}^{*} c_{1}(V)$ are pseudo-effective,
the pull-back $p_{i}^{*}Q$ can be represented by a positive $(2,2)$-current.
Hence, since $S_{\e}$ is a smooth (semi-)positive $(1,1)$-form,
the intersection number $(\{p_{0}^{*} S_{\e} \} \cdot p_{i}^{*}Q)$ is non-negative.
Then, Conclusion (2) follows from
\begin{align*}
(c_{1}(-K_{X_i}) \cdot Q)&=\lim_{\e  \to 0}( (c_{1}(-K_{X_i}) +\e \{\omega_{i}\} + \e O(K_{X_{i}}))   \cdot Q)\\
&=\lim_{\e  \to 0} (\{p_{i*}p_{0}^{*} (S_{\e} + \e P)\} \cdot Q)\\
&=\lim_{\e  \to 0} (\{p_{0}^{*} (S_{\e} + \e P)\} \cdot p_{i}^{*}Q)\\
&=\lim_{\e  \to 0} (\{p_{0}^{*} S_{\e} \} \cdot p_{i}^{*}Q) \geq 0.
\end{align*}
\end{proof}

\begin{cor} \label{push-forward-sing}
We consider Setting \ref{setup} and  the MFS $\varphi\colon  X_{N}=X' \to S$  in Theorem \ref{3-fold-MMP}.
Then, we have$:$

$(1)$ $-4c_1(K_S)-c_1(\Delta)$ is pseudo-effective,
where $\Delta$ is the discriminant divisor of the MFS $\varphi\colon  X \to S$ $($which is a $\mathbb{Q}$-conic bundle$)$.

$(2)$ $\Delta=0$, $c_{1}(K_{S})=0$, and $\kappa(S)=0$ hold;
in particular, $\varphi\colon  X' \to S$ is toroidal over $S$ and $S$ has only rational double points.
Furthermore, when $S$ is smooth,
the variety $X'$ is automatically smooth and $\varphi\colon X' \to S $ is a $($locally trivial$)$ $\PP^1$-bundle.
\end{cor}
\begin{proof}
In the proof, we take Zariski open subsets $S_0 \subset S$ with $\codim(S \setminus S_0) \geq 2$.
To maintain clarity in notation, we consistently refer to these subsets as $S_0$, even though they may vary. 
Since $-K_{X'}$ is $\varphi$-ample by Theorem \ref{3-fold-MMP} (d),
we can take a K\"ahler form $\omega' \in c_{1}(-K_{X'}) + \{\varphi^{*}\omega_{S}\}$,
where $\omega_{S}$ is a fixed K\"ahler form on $S$.
By Proposition \ref{prop-intersection} $(1)$,
there exists a positive current
$$
T_{\e} \in -c_{1}(K_{X'}) + \e \{\omega'  \} =  -(1+\e) c_{1}(K_{X'}) + \e \{ \varphi^{*}\omega_{S} \}
$$
such that $T_{\e}$ is smooth on $\varphi^{-1}(S_{0})$.
Since $X'$ is smooth in codimension $2$, there exists $S_{0}$ such that
$\varphi|_{X_{0}}\colon X_{0}:=\varphi^{-1}(S_{0}) \to S_{0}$  is a conic bundle.
Note that  the Bedford-Taylor product $T_{\e}^{2}$ is defined on $X'_{\reg}$.
By Proposition \ref{push-forward-formula},
the pushforward $\varphi_{*}(T_{\e}^{2})$ defined on $S_{0}$ is a positive current
representing the following class on $S_{0}$:
\begin{align} \label{eq-com}
&\varphi_{*}\big( (-(1+\e) c_{1}(K_{X'}) + \e \{ \varphi^{*}\omega_{S}\})^{2}\big) \notag
\\ =&-(1+\e)^{2}(4c_1(K_S) + c_1(\Delta) )
- 2\e(1+\e) \varphi_{*}  c_{1}(K_{X'}) \cdot \{\omega_{S}\}  + \e^{2}\varphi_{*} \{\varphi^{*} \omega_{S}^{2}\}
\\
=&-(1+\e)^{2}(4c_1(K_S) + c_1(\Delta) )
+ 4\e(1+\e)  \cdot \{\omega_{S}  \}.\notag
\end{align}
Here, we  used that $\varphi_{*}  c_{1}(K_{X'})=-2$ and
$\varphi_{*} \{\varphi^{*} \omega_{S}^{2}\}=0$ hold on $S_{0}. $
Proposition \ref{BEG} shows that $\varphi_{*}(T_{\e}^{2})$
is actually a positive current on $S$
representing the Bott-Chern cohomology class of the right-hand side.
Since the mass measure of $\varphi_*(T_\e^2)$ is uniformly bounded,
we may assume that $\varphi_*(T_\e^2)$ has the weak limit by  weak compactness
(see \cite[(1.14), (1.23) Propositions, Chapter III]{agbook}).
Indeed, the total mass of $\varphi_*(T_\e^2)$ with respect to $\omega_S$
is the intersection number of \eqref{eq-com} with $\{\omega_S\}$,
which  is uniformly bounded in $\e \in [0,1)$. 
Note that $S$ has singularities, but we can apply the weak compactness after taking a resolution of singularities of $S$.
Then, the weak limit of  $\varphi_*(T_\e^2)$ is a positive current  representing the class
$$
\lim_{\e \to 0}-(1+\e)^{2}(4c_1(K_S) + c_1(\Delta) )
+ 4\e(1+\e)  \cdot \{\omega_{S}  \}  = -4c_1(K_S)-c_1(\Delta).
$$
This indicates that  $-4c_1(K_S)-c_1(\Delta)$ is pseudo-effective.

Theorem \ref{3-fold-MMP} (d) shows that $\Delta=0$ and $c_{1}(K_{S})=0$ hold.
The $\mathbb{Q}$-conic bundle $\varphi\colon  X' \to S$ is toroidal by Lemma \ref{MP}
and the singularities of $S$ are rational double points.
Therefore, for the minimal resolution $h \colon \bar {S} \to S$,
we have $K_{\bar {S}}=h^* K_S$, which implies that $\kappa(S)=\kappa(\bar S)=0$.
The latter statement of Conclusion (2) is a special case of \cite[Theorem 5]{ARM14}.
\end{proof}

\begin{rem}
In the proof, if $X'$ is smooth,
the Bedford-Taylor product $T_{\e}^{2}$ can be defined on $X'$
as a positive current representing  $(c_{1}(-K_{i})+\e \{\omega_{i}\})^{2}$.
This is expected to be true even when $X'$ has singularities,
which gives a more direct proof of Proposition \ref{prop-intersection},
but we could not prove this expectation.
We avoided this difficulty
by considering the  $(1,1)$-current $\varphi_{*}(T_{\e}^{2})$  (instead of $T_{\e}^{2}$).
\end{rem}

\section{Proof of the Main Results} \label{Sec4}
This section is devoted to the proof of Theorem \ref{theo-non-proj}.
Throughout this section, let $X$ be a non-projective compact K\"ahler three-fold
with nef anti-canonical bundle.
As explained in Subsection \ref{subsec1-2}, we may assume that $\dim R(X) =2$ for Theorem \ref{theo-non-proj}.
Furthermore, by replacing $X$ with a finite \'etale cover,
we may assume that $\pi_{1}(X) \cong \mathbb{Z}^{\oplus 2q}$,
where $q$ is the irregularity  of $X$.
Here we used the fact that $\pi_{1}(X)$ is almost abelian
(see \cite[Th\'eor\`eme 2]{Pau97} for the proof based on Monge-Amp\`ere equations and
see \cite[Theorem 1.4]{Pau17}, \cite[Theorem 2.2]{Cam95}
for the proof based on variations of K\"ahler-Einstein metrics). 
We consider the case of $q \not= 0$ in Subsection \ref{subsec4-2} and
the case of $q = 0$ in Subsection \ref{subsec4-3}.

\subsection{On the base of MRC fibrations}
Before starting the proof of Theorem \ref{theo-non-proj},
we determine the smooth minimal base of MRC fibrations of $X$.

\begin{prop} \label{prop-mrc}
Let $X$ be a non-projective compact K\"ahler three-fold
with nef anti-canonical bundle such that $\dim R(X)=2$,
where $R(X)$ denotes the smooth minimal base of MRC fibrations of $X$.
Then,
up to a finite \'etale cover of $X$,
the base $R(X)$ is either a torus or a  K3 surface.
In particular, the augmented irregularity of $X$ is  $2$ or $0$.
\end{prop}
\begin{proof}
By replacing $X$ with a finite \'etale cover, we assume that
$\pi_{1}(X) \cong \mathbb{Z}^{\oplus 2q}$.
In particular, we have $q=\dim A(X)$,
where $q$  is the irregularity of $X$.

Consider the same situation as in Theorem \ref{3-fold-MMP}.
Since $X \dashrightarrow X' \to S$ in Theorem \ref{3-fold-MMP} is an  MRC fibration of $X$,
we obtain the minimal resolution $\gamma \colon R(X) \to S$
by noting that $R(X)$ is the smooth minimal base of MRC fibrations and
the base of MRC fibrations is uniquely determined up to bimeromorphic models.
By Corollary \ref{push-forward-sing},
the variety $S$ has cyclic quotient singularities of type $A_{m-1}$ (see Definition \ref{def-Qconic}),
and thus, we have
$$
c_{1}(K_{R(X)}) = c_{1}(\gamma^{*}K_{S} )=0.
$$
The classification of surfaces (e.g.,\,see \cite{BHPV}) implies that
$R(X)$ is a ($2$-dimensional) torus or a K3 surface.
Note that the possibilities of Enriques surfaces and hyperelliptic surfaces
are excluded by the non-projectivity of $R(X)$.

We consider the case of $q=\dim A(X)>0$.
The Albanese map $\alpha\colon  X \to A(X)$ is surjective by \cite[Theorem 1.4]{Pau17},
and thus, so is $\beta \colon S \to A(X)$ in Theorem \ref{3-fold-MMP}.
By pulling back $1$-forms on $A(X)$ using
$\gamma \colon R(X) \to S$ and $\beta \colon S \to A(X)$,
we can see that $R(X)$ is a  torus,
which indicates that $q \geq 2$.
In the case of $q=\dim A(X)=3$, the manifold $X$ is non-uniruled, which contradicts $\dim R(X) =2$.
Thus, we can conclude that $q=2$ and $R(X)$ is a torus.

We consider the remaining case of $q=\dim A(X)=0$.
In this case, the manifold $X$ is simply connected,
and thus,  so is $R(X)$
by $\pi_1(X) \cong \pi_1(R(X))$ (see \cite[Theorem 5.2]{Kol} and \cite[Theorem 4.1]{BC}),
which indicates that $R(X)$ is a K3 surface.
\end{proof}

\subsection{The case of $X$ being non-simply connected }\label{subsec4-2}

In this subsection, we prove Theorem \ref{theo-non-proj}
under assuming that $\pi_{1}(X) \cong \mathbb{Z}^{\oplus 2q}$, $q \not= 0$, and $\dim R(X)=2$.

\begin{theo} \label{thm-alb2}
Consider the same situation as in the beginning of Subsection \ref{subsec4-2}.
Then, up to a finite \'etale cover of $X$,
there exists a numerically flat vector bundle $\mathcal{F}$ on $A(X)$
such that $X$ is isomorphic to the projective space bundle $\mathbb{P}(\mathcal{F})$ over $A(X)$.
\end{theo}
\begin{proof}

We initially show that $\beta\colon S \to A(X)$ in Theorem \ref{3-fold-MMP} is actually isomorphic.
Note that $A(X)$ is a $2$-dimensional torus by Proposition \ref{prop-mrc} and
$\beta \colon  S \to A(X)$ is surjective by \cite[Theorem 1.4]{Pau17}.
Then, we have $K_{S}=\beta^{*}K_{A(X)}+E$
for some effective divisor $E$ supported in the ramified locus of $\beta \colon S \to A(X)$.
By $c_1(K_{A(X)})=0$ and $c_1(K_S)=0$ (see Corollary \ref{push-forward-sing}),
we deduce that $E=0$,
which indicates that $\beta \colon S \to A(X)$ is \'etale and $S$ is a torus.
We can see that $\beta \colon S \to A(X)$ is isomorphic by the universal property of Albanese maps.

By Corollary \ref{push-forward-sing} (2),
the MFS $\varphi\colon  X' \to S \cong A(X)$ in Theorem \ref{3-fold-MMP} is a $\mathbb{P}^{1}$-bundle;
hence, up to a finite \'etale cover of $X'$,
there exists a vector bundle $\mathcal{F}$ on $A(X)$
such that $X' \cong \mathbb{P}(\mathcal{F})$ and $\det(\mathcal{F}) =\mathcal{O}_{A(X)}$
by \cite[Lemma 7.4]{CP91} (see \cite{El82} for more information on Brauer groups).
We show that $\mathcal{F}$ can be chosen to be numerically flat.
We emphasize that the following discussion works only when $X'$ and $S$ are smooth.
By the formula
$$-K_{X'}=\mathcal{O}_{\mathbb{P}(\mathcal{F})}(2)-\varphi^{*}\det \mathcal{F}=\mathcal{O}_{\mathbb{P}(\mathcal{F})}(2),$$
it suffices to show that $-K_{X'}$ is nef.
By applying the regularization theorem \cite[Proposition 3.7]{Dem92}
to the current in Proposition \ref{prop-intersection} (1),
we obtain positive currents $T_\e \in c_1(-K_{X'})+\e \{\omega_{X'}\} $
with analytic singularities such that the singular locus of $T_{\e}$ is not dominant over $S$.
By the proof of Corollary \ref{push-forward-sing},
we have $\varphi_*(c_1(-K_{X'})^2)=0$ and $\lim_{\e \to 0} \varphi_*(T_\e^2)=0$;
hence,  the Lelong number of $T_\e$ uniformly converges to $0$ on $X'$
by \cite[Lemma 15]{Wu22b}.
Thus, we can conclude that $-K_{X'}$ is nef
by the regularization with smooth forms (see \cite[Theorem 1.1]{Dem92},
cf.\,the end of proof of \cite[Theorem 4]{Wu22b} and \cite[Corollary 6]{Wu22b}).

\smallskip

We finally show that $X \dashrightarrow X' $
is actually an isomorphism.
To achieve this, we focus on the final step $X_{N-1} \dashrightarrow X_{N}=X'$ of the MMP $($see \eqref{eq-bir}$)$ and divide the proof into three subsequent claims.

\begin{claim}\label{claim-a1}
$X_{N-1} \dashrightarrow X_{N}=X'$  cannot be a flip.
\end{claim}
\begin{proof}
Every rational curve $R \subset X'$  is contained in a fiber of
$\varphi\colon X' \cong \mathbb{P}(\mathcal{F}) \to S\cong A(X)$.
Consequently, the intersection number $(R \cdot c_{1}(-K_{X'}))=\deg (-K_{\mathbb{P}^{1}})$ is positive.
This implies that $X'$ has no $K_{X'}$-positive rational curves;
therefore $X_{N-1} \dashrightarrow X_{N}=X'$ cannot be a flip
since a flip generates a $K_{X'}$-positive rational curve.
\end{proof}

We consider the possible divisorial contraction $\pi:=\pi_{N-1} \colon X_{N-1} \to X_{N}=X'$
with the exceptional divisor $E$.

\begin{claim}\label{claim-a2}
$X_{N-1} \dashrightarrow X_{N}=X'$ cannot be a divisorial contraction contracting $E$ to the one point.
\end{claim}
\begin{proof}

Since $X_{N-1}$
has  terminal singularities,
we have
\begin{align}\label{eq-can}
-K_{X_{N-1}}=\pi^*(- K_{X_{N}})-aE \text{ \quad for some }a>0.
\end{align}
We can  confirm that the equality $c_{1}(K_{X_{N}})^2 = 0$ is satisfied.
Indeed, the condition $c_{1}(K_{X_{N}})^2 = 0$ remains invariant under finite \'etale covers,
which permits us to assume that  $X'=X_{N}=\mathbb{P}(\mathcal{F})$ 
with  a numerically flat vector bundle $\mathcal{F}$.
Then, the desired equality $c_{1}(K_{X_{N}})^2 = 0$ can be obtained
from $K_{X_N} = \mathcal{O}_{\mathbb{P}(\mathcal{F})}(-2)$ and
$$
c_1(\mathcal{O}_{\mathbb{P}(\mathcal{F})}(1))^2 =
\varphi^{*}(c_1(\mathcal{F})) c_1(\mathcal{O}_{\mathbb{P}(\mathcal{F})}(1)) - \varphi^{*}(c_2(\mathcal{F}) )= 0.
$$
By noting that $E$ is contracted to the one point,
we can see that $c_1(-K_{X_{N-1}})^2 = a^2 c_1(E)^2$ by $c_{1}(K_{X_{N}})^2 = 0$. 
Then, since $\OX_E(-E)$ is ample, we obtain
$$
(c_1(K_{X_{N-1}})^2 \cdot \{\omega_{X_{N-1}}\})=
a^2 (c_{1} (E)|_{E} \cdot \{\omega_{N-1}\}|_E)<0.
$$
This contradicts Proposition \ref{prop-intersection} (2).
\end{proof}

\begin{claim}\label{claim-a3}
$X_{N-1} \dashrightarrow X_{N}=X'$ cannot be
a divisorial contraction contracting $E$ to the curve $C:=\pi(E)$.
\end{claim}
\begin{proof}
Assume that there exists a surface $V \subset X_{N}$ such that
$K_{X_{N}}|_{V}$ is numerically trivial and $\bar {V} \cap E$  is an effective curve,
where $\bar {V} \subset X_{N-1}$ is the strict transform of $V$.
Then, by \eqref{eq-can}, we can see that
$$
(c_1(-K_{X_{N-1}}) \cdot c_1(\bar {V}) \cdot\omega_{N-1})  =
a (c_1(-E) \cdot c_1(\bar {V}) \cdot\omega_{N-1})<0.
$$
This contradicts Proposition \ref{prop-intersection} (2).

To find such a surface $V$, we show that $C$ is a fiber of $\varphi\colon  X_{N}=X' \to A(X)$,
which follows since $X \to A(X)$ is a smooth fibration
outside a Zariski closed set of $A(X)$ of codimension  $\geq 2$ by \cite[Proposition 1.6]{Cao13}.
Indeed, otherwise, the image $\varphi(C)$ is a curve in $A(X)$.
Then, a fiber of $X_{N-1} \to X' \to A(X)$ at a general point $p \in \varphi(C)$
has at least two irreducible components:
the strict transform of the fiber $F$ of $\varphi\colon  X' \to A(X)$ at $p$
and the inverse image of $F \cap C$.
This contradicts that $X \to A(X)$ is smooth in codimension $1$.
Meanwhile, by \cite[Theorem 1.18]{DPS94},
the numerically flat vector bundle $\mathcal{F}$ is constructed from a $\rm{GL}$-representation of $\pi_1(S)$.
Furthermore, since $\pi_1(S)$ is abelian,
this representation is the direct sum of $1$-dimensional representations.
Hence, there exist Hermitian flat line bundles $L_1, L_2$ on $S$
such that
$$0 \to L_1 \to \mathcal{F} \to L_2 \to 0.$$
Define the surface $V \subset X'=X_{N}$ by the image of $\PP(L_2) \subset \mathbb{P}(\mathcal{F})$ via 
$\mathbb{P}(\mathcal{F}) \to X'=X_{N}$.
Then, we can easily see that $K_{X_{N}}|_{V}$ is numerically trivial and
the intersection $\bar {V} \cap E$ is the fiber $\pi \colon X_{N-1} \to X'=X_{N}$
at the non-empty $0$-dimensional varieties $V\cap C$.
\end{proof}
This finishes the proof of Theorem \ref{theo-non-proj}
under assuming that $\pi_{1}(X) \cong \mathbb{Z}^{\oplus 2q}$, $q \not= 0$, and $\dim R(X)=2$.
\end{proof}

\begin{rem}\label{rem-cao}
The final case (where $E$ is contracted to the curve $C$)
can be actually excluded by another approach without using \cite[Proposition 1.6]{Cao13}
(see the final step of the proof of Theorem \ref{thm-pro}).
However, this approach explained in Theorem \ref{thm-pro} is quite complex, and thus
we provide a more straightforward proof here using \cite[Proposition 1.6]{Cao13}.
\end{rem}

\subsection{The case of $X$ being simply connected}\label{subsec4-3}

In this subsection, we prove Theorem \ref{theo-non-proj}
under assuming that $\pi_{1}(X)=\{\id\}$ and $\dim R(X) =2$.
Compared to Subsection \ref{subsec4-2},
a notable challenge is that $S$ may have singularities,
which prevents us from employing the same argument as in Theorem \ref{thm-alb2}.
The following example helps us to understand this difficulty.

\begin{ex}
For a $2$-dimensional torus $A$,
we consider
$$X':=(\PP^1 \times A)/\mu_2 \to S=A/\mu_2, $$
where $\mu_2:=\mathbb{Z}/2\mathbb{Z}$  acts on  $\PP^1 \times A$
by $(-1) \cdot (t, z_1,z_2)=(-t,-z_1,-z_2)$.
Both $S$ and $X'$ are simply connected and
$\varphi\colon  X' \to S$ is a $\mathbb{Q}$-conic bundle such that $-K_{X'}$ is nef.
Nevertheless, the fibration $\varphi\colon  X' \to S$ is not even locally trivial.
\end{ex}

In fact, the above example never appears as an outcome of the MMP of $X$ with nef anti-canonical bundle.
More precisely, the following theorem can be proved:
\begin{theo}\label{thm-pro}
Consider the same situation as in the beginning of Subsection \ref{subsec4-3}.
Then, the manifold $X$ is isomorphic to the product of a K3 surface and the projective line $\PP^1$.
\end{theo}
\begin{proof}
By \cite[Corollary 6.7]{Cam04b},
there exists a quasi-\'etale cover $\tau \colon  S^{\dagger} \to S$
such that $S^{\dagger}$ is either a torus or a normal K3 surface
(i.e.,\,a normal surface whose minimal resolution is a K3 surface)
and $\tau \colon  S^{\dagger} \to S$ is an orbifold morphism
(i.e.,\,it can be locally described  as $\tau\colon U^\dagger/G^\dagger \to U/G$
induced by a morphism $\hat \tau\colon U^\dagger \to U$  and a group homomorphism $\rho\colon G^\dagger \to G$
such that $\hat \tau(g z)=\rho(g) \hat \tau(z)$ holds for any $g \in G^\dagger$ and $z \in U^\dagger$,
where $U^\dagger$ (resp.\,$U$) is a local smooth ramified cover
with the linear action of the finite group $G^\dagger$ (resp.\,$G$).)
Note that any normal K3 surface is simply connected.
We consider the base change:
\[
\xymatrix{
X^\dagger \ar[d]_{\varphi^\dagger} \ar[r]^{\nu} & X' \ar[d]^{\varphi} \\
S^\dagger \ar[r]_{\tau} & S
}
\]

The fibration $\varphi \colon X' \to S$
satisfies the assumptions of Theorem \ref{theo-flat}
when we set $Y_{0}:=S_{\reg}$ and $L:=-K_{X'}$.
Indeed, Conditions (a), (b), (c) are confirmed by Proposition \ref{prop-intersection},
Corollary \ref{push-forward-sing}, and Theorem \ref{3-fold-MMP}.
Furthermore, Condition (d) can be confirmed by orbifold structures as follows:
Indeed, the reflexive sheaf $\mathcal{V}_{p}$ defined by $L=-K_{X'}$ as in Theorem \ref{theo-flat}
is an orbifold vector bundle on the orbifold $S$.
By \cite[Lemma 1]{Wu23}, there exists a continuous function $\psi$ on $S$
whose pull-back on each local smooth ramified cover is smooth such that
$\omega_S+ \sqrt{-1} \partial \overline{\partial} \psi$ defines an orbifold K\"ahler form on $S$.
For a smooth orbifold metric $g$ on $\det \mathcal{V}_p$,
we take  $C \gg 1$ such that
$\sqrt{-1}\Theta_{g} +C(\omega_S+ \sqrt{-1} \partial \overline{\partial} \psi)$
is positive on each local smooth ramified cover.
Then $g e^{-C\psi}$ is a metric satisfying Condition (d).
Consequently, Theorems \ref{theo-flat} and \ref{thm-flat} show that
$\mathcal{V}_p$ is a numerical flat orbifold vector bundle on $S$.
In the same way, we deduce that the fibration $\varphi^{\dagger} \colon X^{\dagger} \to S^{^{\dagger}}$
also satisfies the assumptions of Theorem \ref{theo-flat}
when we set $Y_{0}:=S^{\dagger} \setminus \tau^{-1}(S_{\sing})$ and $L=-K_{X^{\dagger}}$,
by noting that  $\nu\colon X^{\dagger} \to X'$ is a quasi-\'etale cover and
$-K_{X^{\dagger}}$ is $\varphi^{\dagger}$-ample by the construction of
$\varphi^{\dagger} \colon X^{\dagger} \to S^{\dagger}$.
Consequently, the sheaf $\mathcal{V}^{\dagger}_p $ defined by $L=-K_{X^{\dagger}}$
is also a numerical flat orbifold bundle on $S^{\dagger}$.

We divide our situation into Case 1 and Case 2,
and respectively prove Claim \ref{claim-st1} and Claim \ref{claim-st2}.
\smallskip
\\
\quad {\bf{Case 1}}: The case where $S^{\dagger}$ is a ($2$-dimensional) torus. \\
\quad{\bf{Case 2}}: The case where $S^{\dagger}$ is a normal K3 surface.

\begin{claim}\label{claim-st1}
In Case 1, up to a finite \'etale cover of  $S^{\dagger}$,
the variety $X^\dagger$ is isomorphic to the projective space bundle $\PP(\mathcal{F})$
of a numerically flat vector bundle $\mathcal{F}$ over the torus $S^{\dagger}$.
Moreover, the vector bundle $\mathcal{F}$ admits a filtration by Hermitian flat line bundles$:$
$$0 \to L_1 \to \mathcal{F} \to L_2 \to 0.$$
\end{claim}
\begin{proof}
Since $S^{\dagger}$ is smooth,
the sheaf $\mathcal{V}^{\dagger}_p $ defined as above
is a numerical flat vector bundle on the torus $S^{\dagger}$ by \cite[Main Theorem]{Wu22b}.
Then, \cite[Proposition 2.5, Remark 2.6 (b)]{MW} shows that
$\varphi^{\dagger}\colon X^{\dagger} \to S^{\dagger}$
is locally constant.
In particular, the variety $X^{\dagger}$ is smooth and
$\varphi^{\dagger}\colon X^{\dagger} \to S^{\dagger}$
is a $\mathbb{P}^{1}$-bundle.
Then, the first conclusion follows from the same argument as in Theorem \ref{thm-alb2}.

The flat vector bundle $\mathcal{F}$ is constructed from
a $\rm{GL}$-representation of the fundamental group $\pi_{1}(S^{\dagger})$.
Thus, since $\pi_{1}(S^{\dagger})$ is abelian,
the vector bundle $\mathcal{F}$ admits the desired filtration.
\end{proof}

\begin{claim}\label{claim-st2}
In Case 2, the variety $X^\dagger$ is isomorphic to the product $S^{\dagger} \times \PP^1$.
\end{claim}
\begin{proof}

We initially show that $\pi_1(X^{\dagger}_{\reg})=\{\id\}$ by applying the Van Kampen theorem.
Note that any normal K3 surface is simply connected.
The variety $S^{\dagger}$ is the universal cover in the sense of orbifolds (see \cite[D\'efinition 5.3]{Cam04b});
hence, we  obtain $\pi_1(S^{\dagger}_\reg) \cong \pi_1(S^{\dagger})=\{\id\}$.
Since $\varphi \colon X' \to S$ is a smooth $\PP^1$-bundle on $S_\reg$ (which is preserved under the base change),
we deduce that $\pi_1((\varphi^{\dagger})^{-1}(S^{\dagger}_\reg))=\{\id\}$.
Near a singular point in $S_{\sing}$ at which $\tau \colon S^{\dagger} \to S$ is not \'etale,
both $\varphi\colon  X' \to S$ and  $\varphi^{\dagger} \colon  X^{\dagger}  \to S^{\dagger}$
can be locally described as follows:
\[
\xymatrix{
(\mathbb{P}^1 \times \mathbb{C}^2)/\mu_{m^\dagger} \ar[d]_{\varphi^\dagger} \ar[r]^{\nu} 
& (\mathbb{P}^1 \times \mathbb{C}^2)/\mu_m \ar[d]^{\varphi} \\
\mathbb{C}^2/\mu_{m^\dagger} \ar[r]_{\tau} 
& \mathbb{C}^2/\mu_m
}
\]
Here the action of $\e \in \mu_{m^{\dagger}}$
is given by $\e \cdot (t, z_1,z_2)=(\e^b t, \e z_1, \e^{-1} z_2)$ and
$\tau$ is given by $\CC^2/\mu_{m^{\dagger}} \to \CC^2/\mu_m$
induced by the inclusion $\mu_{m^{\dagger}} \to \mu_m$.
To apply the Van Kampen theorem, we regard $X^{\dagger}_\reg$ as the union of
$(\varphi^{\dagger})^{-1}(S^{\dagger}_\reg)$
and open neighborhoods $V_{i}$ of $(\varphi^{\dagger})^{-1}(S^{\dagger}_\sing)$ in $X^{\dagger}_\reg$.
Then, we can see that
\begin{itemize}
\item[$\bullet$] $V_{i}$  is homeomorphic to
$(\PP^1 \times \CC^2 \setminus \{(0,\boldsymbol{0}),(\infty,\boldsymbol{0})\})/\mu_{m\dagger}$,
\item[$\bullet$]  $(\varphi^{\dagger})^{-1}(S^{\dagger}_\reg) \cap V_{i}$
is homeomorphic to $(\PP^1 \times (\CC^2 \setminus \{\boldsymbol{0}\}) \setminus \{(0,\boldsymbol{0}),(\infty,\boldsymbol{0})\})/\mu_{m\dagger}$,
\end{itemize}
where $\boldsymbol{0}:=(0,0) \in \mathbb{C}^{2}$.
Consider the homomorphism induced by the natural inclusion:
$$
\pi_{1}(\PP^1 \times (\CC^2 \setminus \{\boldsymbol{0}\}) \setminus
\{(0,\boldsymbol{0}),(\infty,\boldsymbol{0})\}/\mu_{m\dagger})
\to \pi_{1}(\PP^1 \times \CC^2 \setminus \{(0,\boldsymbol{0}),(\infty,\boldsymbol{0})\}/\mu_{m\dagger}).
$$
Since $\PP^1 \times (\CC^2 \setminus \{\boldsymbol{0}\}) \setminus \{(0,\boldsymbol{0}),(\infty,\boldsymbol{0})\}$ and
$\PP^1 \times \CC^2 \setminus \{(0,\boldsymbol{0}),(\infty,\boldsymbol{0})\}$ are the universal covers respectively,
the above homomorphism can be regarded as the identity map of $\mu_{m\dagger}$.
Therefore, the Van Kampen theorem shows that  $\pi_1(X^{\dagger}_\reg)=\{\id\}$.

The sheaf $\mathcal{V}^{\dagger}_p $ defined by  $L:=-K_{X^{\dagger}}$
is a numerical flat orbifold vector bundle on  $S^{\dagger}$.
Since $\varphi^{\dagger}\colon X^{\dagger} \to S^{\dagger}$ is an orbifold morphism,
the sheaf $(\varphi^{\dagger *} \mathcal{V}^{\dagger}_p)^{**}$ is also
a numerically flat orbifold vector bundle.
By Lemma \ref{lem-topo} (which is proved later),
we have
$$H^1(X^{\dagger}_\reg, \OX_{X^{\dagger}_\reg})\cong H^1(X^{\dagger}, \OX_{X^{\dagger}})=0.$$
Then, Corollary \ref{cor-flat} shows that $(\varphi^{\dagger *} \mathcal{V}^{\dagger}_p)^{**}$
is a trivial vector bundle on $X^{\dagger}$.
Note that the  analogue of Lemma \ref{lem-topo} in the $2$-dimensional case is false,
which is the reason to consider the pull-back of $\mathcal{V}^{\dagger}_p$ to $X^\dagger$.
Since $\mathcal{V}^{\dagger}_p$ is locally free on $S^{\dagger}_{\reg}$,
the projection formula indicates that
$\mathcal{V}^{\dagger}_p \cong \varphi^{\dagger}_{*}((\varphi^{\dagger *} \mathcal{V}^{\dagger}_p)^{**})$
holds on $S^\dagger_{\reg}$;
hence $\mathcal{V}^{\dagger}_p$  is a trivial vector bundle on $S^\dagger_{\reg}$.
By  \cite[Proposition 2.5]{MW} and $\pi_1(S^\dagger_\reg)=\{\id\}$,
we conclude  that $\varphi^{\dagger}\colon X^{\dagger} \to S^{\dagger}$
gives the product structure $S^\dagger_\reg \times \PP^1$ over $S^\dagger_\reg$.
Precisely speaking, we need to check Condition  (2) in \cite[Proposition 2.5]{MW},
that is, the natural morphism
$\Sym^m \mathcal{V}^{\dagger}_p \to \mathcal{V}^{\dagger}_{pm}$
on $S^\dagger_\reg$  is a morphism of local systems.
In our case, this is automatically satisfied
since the reflexive hulls of $\Sym^m \mathcal{V}^{\dagger}_p$ and $\mathcal{V}^{\dagger}_{pm}$
are trivial vector bundles on $S^\dagger$ (cf.\,\cite[Remark 2.6 (b)]{MW}). 

If the meromorphic map $S^{\dagger} \times \PP^1 \dashrightarrow X^{\dagger}$
(obtained from the above product structure over $S^{\dagger}$) fails to be an isomorphism,
the fiber $\PP^1$ at a singular point of $S^\dagger$ would be contractible by \cite[Proposition 2.1.13]{Kol91}.
However, such a contraction does not exist.
Therefore, we can conclude that $X^{\dagger}=S^{\dagger} \times \PP^1$.
\end{proof}

We finally show that $\pi\colon  X \dashrightarrow X'$ is actually isomorphic.
This finishes the proof of Theorem \ref{thm-pro}.
Indeed, the MFS $\varphi \colon X' \to S$ is a conic bundle
by noting that $\varphi \colon X' \to S$ is toroidal at any point.
The fibration $\varphi\colon  X \cong X' \to S$ is a locally constant $\mathbb{P}^{1}$-bundle
by the same argument as in Theorem \ref{thm-alb2}.
Meanwhile, since $\pi_1(X) \cong \pi_1(S)$ holds
by \cite[Theorem 5.2]{Kol} and \cite[Theorem 4.1]{BC},
the base $S$ is simply connected,
which implies that  $X \cong X'$ is the product of $S$ and $\mathbb{P}^{1}$.

\medskip

To verify that  $\pi\colon  X \dashrightarrow X'$ is isomorphic,
we divide our situation into the four subsequent claims as in the proof of Theorem \ref{thm-alb2}.

\begin{claim}\label{claim-b1}
In both Case 1 and Case 2,
the final step $X_{N-1} \dashrightarrow X_{N}=X'$ of the MMP  cannot be a flip.
\end{claim}
\begin{proof}
The same strategy as in Theorem \ref{thm-alb2} works.
Let $R$ be a rational curve on $X'$.
Let $d$ be the degree of the restriction $\nu^{-1}(R) \to R$ of $\nu \colon  X^{\dagger} \to X'$.
Then, we see that
$$(c_{1}(K_{X'}) \cdot R)=\frac{1}{d} (c_{1}(K_{X^{\dagger}}) \cdot  \nu^{-1}(R)) \leq 0.$$
This implies that $\pi\colon  X_{N-1} \dashrightarrow X_{N}=X'$ cannot be a flip.
\end{proof}

\begin{claim}\label{claim-b2}
In both Case 1 and Case 2,
the final step $\pi\colon  X_{N-1} \dashrightarrow X_{N}=X'$ cannot contract a surface to the one point
\end{claim}
\begin{proof}
Let $d$ be the degree of $\nu\colon  X^{\dagger} \to X'$. 
In both Case 1 and Case 2, we have $K_{X'}^2=(1/d )\nu_* K_{X^{\dagger}}^2=0$ 
since $K_{X^{\dagger}}^2=0$ holds and $\nu\colon  X^{\dagger} \to X' $ is an orbifold morphism.
Then, by the same argument as in Claim \ref{claim-a2},
we obtain a contradiction.
\end{proof}

\begin{claim}\label{claim-b3}
In Case 2,
the final step $\pi\colon  X_{N-1} \dashrightarrow X_{N}=X'$ cannot contract a surface $E$ to a curve $C$.
\end{claim}
\begin{proof}
In Case 2, for a general $ t \in \mathbb{P}^{1}$, the surface $V:=\nu(S^{\dagger} \times \{t\}) \subset X'$  such that
$K_{X_{N}}|_{V}$ is numerically trivial and $\bar {V} \cap E$  is an effective  curve,
where $\bar {V} $ is the strict transform of $V$.
Then, by the same argument as in the first paragraph of Claim \ref{claim-a3}
we obtain a contradiction.
\end{proof}

\begin{claim}\label{claim-b4}
In Case 1,
the final step $\pi\colon  X_{N-1} \dashrightarrow X_{N}=X'$ cannot contract a surface $E$ to a curve $C$.
\end{claim}
\begin{proof}

We initially consider the case where $C$ intersects with the image $V:=\nu(\PP(L_2))$. 
In this case, we can see that $K_{X_{N}}|_{V}$ is numerically trivial  and $\bar {V} \cap E$  is an effective curve, 
where $\bar {V} \subset X_{N-1}$ is the strict transform of $V$. 
Hence, the same argument as in the first paragraph of Claim \ref{claim-a3} yields a contradiction.
Thus, we may assume that $C$ does not intersect with $\nu(\PP(L_2))$.

We will prove that after the base change by $\nu\colon X^{\dagger} \to X$,
the contraction $\pi_{N-1} \colon  X_{N-1} \to X_{N}=X'$
coincides with the blow-up of  $X_{N}=X'$ along $C$, which will lead to a contradiction.
To confirm this, we consider the following diagram:
\[
\xymatrix{
\mathbb{P}(\mathcal{F}|_{T^\dagger}) \ar[r]^{\text{hook}} \ar[d]_{\varphi^\dagger} 
& \mathbb{P}(\mathcal{F}) = X^\dagger \ar[d]_{\varphi^\dagger} \ar[r]^{\nu} 
& X' = X_N \ar[d]^{\varphi} 
& & X_{N-1} \ar[ll]_{\pi := \pi_{N-1}} \\
T^\dagger := \varphi^\dagger(C^\dagger) \ar[r]^{\text{hook}} 
& S^\dagger \ar[r]^{\tau} 
& S = S^\dagger / G & &
}
\]
where $C^{\dagger}:=\nu^{-1}(C)$. 

We now prove that $T^{\dagger}$
is the disjoint union of elliptic curves.
Consider the irreducible decomposition $T^{\dagger}=\cup_{i} T_{i}$.
By \cite[Lemma 10.8]{Uen}, there exists an abelian variety in $S^\dagger$ containing $T_{i}$.
Since $S^\dagger$ is a non-projective torus of dimension $2$,
the curve $T_{i}$ must be an elliptic curve.
If $T_{i} \cap T_{j} \not= \emptyset$ for some $i \not= j$,
there exists $g \in G$  such that $T_{j}=gT_{i}$ holds  by $\tau(T_{i})=
\varphi\circ \nu(C^{\dagger})=\tau(T_{j})$.
Then, up to a translation of $T_{i}$ and $gT_{i}$, 
we can see that $gT_{i} \to S^\dagger \to S^\dagger/T_{i} $ is an isogeny.
By \cite[Proposition 6.1, Chap I]{BL}, up to an isogeny of $S^\dagger$,
the torus $S^\dagger$ is the product of elliptic curves.
In particular $S^\dagger$ is projective, which contradicts the assumption that $S^\dagger$ is a non-projective torus.

Subsequently, we prove that $C^{\dagger}$ is also the disjoint union of elliptic curves. 
For simplicity, we assume that $T^{\dagger}$ is connected (equivalently, $T^{\dagger}$ is irreducible). 
Note that, in the general case, applying the same argument to each connected component of $T^{\dagger}$ yields the desired conclusion. 
To this end, we write the  divisor $C^{\dagger}$ on $\PP(\mathcal{F}|_{T^{\dagger}})$ 
as 
\begin{align}\label{eq-c}
C^{\dagger} = \OX_{\PP(\mathcal{F}|_{T^{\dagger}})}(a) + \varphi^{\dagger*}D
\end{align}
for some $a \in \mathbb{Z}$ and some line bundle $D$ on $T^{\dagger}$, 
and prove that $a>0$ and $D$ is numerically trivial. 
By considering the intersection number with a general fiber of
$\PP(\mathcal{F}|_{T^{\dagger}}) \to T^{\dagger}$,
we can see that $a\geq 0$ since $\OX_{\PP(\mathcal{F}|_{T^{\dagger}})}(1)$ is nef.
Meanwhile, by the exact sequence in Claim \ref{claim-st1}, 
we can see that $c_{1}(\OX_{\PP(\mathcal{F})}(1)) = c_{1}(\PP(L_{2}))$, 
where $\PP(L_{2})$ is regarded as a divisor on $\PP(\mathcal{F})$. 
Hence, by the assumption that $C \cap \nu (\PP(L_2))= \emptyset$,
we have 
\begin{align}\label{eq-inter}
  (C^{\dagger} \cdot \OX_{ \PP( \mathcal{F}|_{T^{\dagger}})}  (1))
=(C^{\dagger} \cdot \OX_{\PP(\mathcal{F})}(1) )=0.
\end{align}
By using \eqref{eq-inter} and $(\OX_{\PP(\mathcal{F}|_{T^{\dagger}})}(1)
\cdot \OX_{\PP(\mathcal{F}|_{T^{\dagger}})}(1))=0$,
we deduce that $c_{1}(D)=0$.

We consider two cases, depending on whether the following exact sequence, 
obtained from restricting the exact sequence in Claim \ref{claim-st1}, splits or not:
$$0 \to L_1|_{{T^{\dagger}}} \to \mathcal{F}|_{{T^{\dagger}}} \to L_2|_{{T^{\dagger}}} \to 0. $$
We first consider the case where this sequence splits on the elliptic curve $T^{\dagger}$. 
Take a lattice $\Lambda$ of $\mathbb{C}$ such that $T^{\dagger} \cong \mathbb{C}/\Lambda $. 
Then, since $L_1|_{{T^{\dagger}}}$, $L_2|_{{T^{\dagger}}}$, and $D$ are numerically trivial line bundles, 
we can take the corresponding unitary representations
$$\rho_1,\rho_2,\rho_D \colon \Lambda \to {\rm{U}}(1).$$
Take a section 
$$s \in H^{0}(\PP(\mathcal{F}|_{T^{\dagger}}),  \OX_{\PP(\mathcal{F}|_{T^{\dagger}})}(a) + \varphi^{\dagger*}D) \cong 
H^{0}(T^{\dagger},  \mathrm{Sym}^a(L_1|_{T^{\dagger}} \oplus L_2|_{T^{\dagger}}) \otimes D)$$
whose divisor $\rm{div}(s)$ coincides with $C^{\dagger}$. 
By pulling back $s$ to the universal cover $\mathbb{C} \times \mathbb{P}^{1}$, 
the section $s$ can be identified with an element 
$$\sum_{i,j \geq 0, i+j=a} f_{ij}(z)u^i v^j \in \mathcal{O}(\mathbb{C})[u,v]$$ 
satisfying that 
$$f_{ij}(z+\lambda)=f_{ij}(z)\rho_1(\lambda)^i \rho_{2}(\lambda)^j \rho_D(\lambda) \text{ for any } \lambda \in \Lambda,$$
where $(z, [u : v])$ is a coordinate of $\mathbb{C} \times \mathbb{P}^1$. 
The Liouville theorem shows that $f_{ij}$ are constant functions on $\mathbb{C}$. 
This indicates that $C^\dagger$ is smooth and is the disjoint union of elliptic curves. 
Indeed,  the inverse image of $C^{\dagger}$ in $\mathbb{C} \times \mathbb{P}^{1}$ 
can be written as 
$$\{(z, [u : v])\in \mathbb{C} \times \mathbb{P}^1\,|\, \sum_{i,j \geq 0, i+j=a} f_{ij}u^i v^j =0 \},$$
and $C^\dagger$  is the quotient of this inverse image.

We consider the remaining case where the above exact sequence does not split. 
In this case, by considering the extension class, we can see that $L_1|_{T^{\dagger}}=L_2|_{T^{\dagger}}$ and 
that $\PP(\mathcal{F}|_{T^{\dagger}})$ is the same  as in \cite[Example 1.7]{DPS94}. 
By \cite[Example 1.7]{DPS94}, we can see that 
$c_1(\OX_{\PP(\mathcal{F}|_{T^{\dagger}})}(a))$
contains the only positive current associated with an effective curve 
whose support is 
$\PP(L_{2}|_{T^{\dagger}}) $.
This  indicates that
$C^{\dagger} =\PP(L_{2}|_{T^{\dagger}}) $
by $c_1(C^{\dagger})=c_1( \OX_{\PP(\mathcal{F}|_{T^{\dagger}})}(a))$, 
and thus $C^{\dagger}$ is an elliptic curve.

\smallskip

Consider the normalization $X_{N-1}^\dagger$ of the fiber product $X^\dagger \times_{X'} X_{N-1}$:
\[
\xymatrix{
X_{N-1}^\dagger \ar[d]_{\pi^\dagger} \ar[r]^{\nu^\dagger} 
& X_{N-1} \ar[d]^{\pi := \pi_{N-1}} \\
X^\dagger \ar[r]^{\nu} 
& X'
}
\]
We prove that $$\pi^{\dagger}\colon  (X_{N-1}^{\dagger}, E^\dagger:=\nu^{\dagger*}E) \to (X^{\dagger}, C^\dagger)$$ satisfies
all the conditions of Lemma \ref{divisorial_con} (which is proved later)
to conclude that it coincides with the blow-up of $X^\dagger$ along  $C^{\dagger}$.
Condition (0) of Lemma \ref{divisorial_con} is satisfied since $C^{\dagger}$ and $X^{\dagger}$ are smooth.
Additionally, Condition (3) is also satisfied
since  $X_{N-1}$ has isolated singularities and $\nu\colon X^{\dagger} \to X'$ is ramified along only $X'_{\sing}$
(which are finitely many points).
As was the projective case, the contraction $\pi\colon X_{N-1} \to X'$ satisfies Conditions (1), (2)
and both $-E$ and $-K_{X_{N-1}}$ are $\pi$-ample (see  \cite[Lemma 7.8]{HP15a} for (2)
and \cite{HP15a, HP15b, HP16}).
By noting that $\nu\colon X^{\dagger} \to X'$ is ramified along only $X'_{\sing}$,
we can see that $\pi^{\dagger}\colon (X_{N-1}^{\dagger}, E^\dagger) \to (X^{\dagger}, C^\dagger)$
also satisfies Conditions (1), (2) and both $-E^\dagger=-\nu^{\dagger *} E$ and 
$-K_{X_{N-1}^\dagger}=-\nu^{\dagger *} K_{X_{N-1}}$  on $X_{N-1}^\dagger$ are $\pi^\dagger$-ample
outside $\nu^{-1}(X'_{\sing})$.
To check that $-E^\dagger=-\nu^{\dagger *} E$ is $\pi^\dagger$-nef,
we take a fiber $F$ of $X_{N-1}^\dagger \to X^\dagger$.
Since $\nu^{\dagger } (F)$ is a curve contracted by $\pi \colon X_{N-1} \to X'=X_{N}$,
we have
$$
(-E^\dagger \cdot F)=(-\nu^{\dagger *} E \cdot F)=(-E \cdot \nu^{\dagger } (F))>0
$$
which implies that $-E^\dagger$ is $\pi^\dagger$-nef.
Hence, Conditions (4), (5) are satisfied. 
Lemma \ref{divisorial_con} shows
that $\pi^{\dagger}\colon (X_{N-1}^{\dagger}, E^\dagger) \to (X^{\dagger}, C^\dagger)$ is actually
the blow-up of $X^\dagger$ along  $C^{\dagger}$.

Now, let us compute intersection numbers and derive a contradiction.
By the formula $K_{X^\dagger_{N-1}}=\pi^{\dagger *} K_{X^\dagger}+E^\dagger$,
we have $$
(K_{X^\dagger_{N-1}}^2 \cdot \omega_{})=
(\pi^{\dagger*} K_{X^\dagger}^2 \cdot \omega)+
2 (\pi^{\dagger*} K_{X^\dagger}\cdot E^\dagger \cdot \omega)
+((E^\dagger)^2 \cdot\omega), 
$$
where $\omega$ is a K\"ahler form on $X^\dagger_{N-1}$.
The left-hand side is non-negative. 
Indeed, for $T_{\e} \in c_{1}(-K_{X_{N-1}}) + \e \{\omega_{N-1}\}$ in Proposition \ref{prop-intersection} (1),
the pull-back $\nu^{\dagger*}(T_{\e} )$  represents
$c_{1}(-K_{X^{\dagger}_{N-1}}) + \nu^{\dagger*}\e \{\omega_{N-1}\}$ and
is smooth outside a Zariski closed set of codimension $\geq 2$
since $\nu^{\dagger} \colon X^{\dagger}_{N-1} \to X_{N-1}$ is quasi-\'etale.
Thus, by noting that $X^{\dagger}_{N-1}$ is smooth,
we see that the Monge-Amp\`ere operator $\nu^{\dagger*}(T_{\e} )^{2}$ is well-defined.
This implies that $
(K_{X^\dagger_{N-1}}^2 \cdot \omega) \geq 0. 
$
Meanwhile, the right-hand side is negative. 
Indeed, the first term is zero by $K_{X^\dagger}^2=0$.
Furthermore, the second term  is also zero 
since $K_{X^\dagger}|_{C^\dagger}$ is numerically trivial 
by
$$
K_{X^\dagger}|_{C^\dagger}=(K_{X^\dagger} |_{\mathbb{P} ( \mathcal{F} |_{T^{\dagger}} ) }) |_{C^\dagger}
=K_{\mathbb{P} ( \mathcal{F} |_{T^{\dagger}} ) }|_{C^\dagger}=0.$$
Since $\OX(-E^\dagger)|_{E^\dagger}=\OX_{\PP(N_{C^\dagger/X^\dagger_{}})}(1)$  holds
and $N_{C^\dagger/X^\dagger_{}}$ is numerically flat,
we can see that for $\pi^{\dagger}|_{E^\dagger} \colon E^\dagger \to C^{\dagger}$, 
\begin{align*}
((E^\dagger)^2 \cdot \omega)
&=(\OX_{\PP(N_{C^{\dagger}/X^\dagger_{}})}(-1) \cdot \omega|_{E^\dagger})
\\ &=-\int_{C^\dagger} (\pi^\dagger|_{E^\dagger})_*( \omega|_{E^\dagger})<0.
\end{align*}
This is a contradiction.
\end{proof}
This finishes the proof of Theorem \ref{thm-pro}, 
and thus completes  the proof of Theorem \ref{theo-non-proj} 
under assuming that $\pi_{1}(X)=\{\id\}$ and $\dim R(X) =2$.
\end{proof}

In the following, we give the proofs of the two lemmas used in the proof.
Lemm \ref{lem-topo} is  an easy variant of  \cite[Lemma 4]{Wu22} and
Lemma \ref{divisorial_con} is  a K\"ahler counterpart of \cite[Proposition 1.2]{Tzi03}.

\begin{lemm} \label{lem-topo}
Let $X$ be a $($not necessarily compact$)$  analytic variety of dimension $3$ and
let $E$ be a vector bundle on $X$.
Assume that $X$ has isolated cyclic quotient singularities$:$
more precisely, near a singular point,  there exists $m \in \mathbb{Z}+$
such that $X \cong \mathbb{D}/\mu_m$,
where $\mathbb{D}$ is a polydisc centered at the origin and
the action of $\mu_m$ is given by
$$\e \cdot (z_0,z_1, z_2)=(\e^{p_0}z_0,\e^{p_1}z_1, \e^{p_2}z_2) \text{ for } \e \in \mu_m
$$
for some $p_i \in \ZZ$ with isolated singular point.
Then, the morphism
$H^1(X, E)\to H^1(X_\reg, E)$
induced by the restriction morphism is isomorphic.
\end{lemm}
\begin{proof}
The germ of $X$ at any point $x$ is Cohen-Macaulay.
In particular, the depth of $\OX_{X,x}$ (i.e.,\,the length of maximal regular sequences contained in the maximal ideal) is equal to the Krull dimension of $\OX_{X,x}$.
Since $E$ is locally free, the depth of $E_x$ at any point $x$ is $3$.
By \cite[Theorem 1.14]{ST06}, the morphism
$H^1(X, E)\to H^1(X_\reg, E)$  is isomorphic.
\end{proof}

\begin{lemm}
\label{divisorial_con}
Let $f \colon  (Y, E) \to  (X, C)$ be a divisorial contraction between
normal K\"ahler spaces $X$ and $Y$ of dimension $3$ such that $E$ is an irreducible surface
and $f(E)=C$ is a curve.
Assume the following conditions$:$
\begin{itemize}
\item[(0)] $f(E)=C$ is a smooth curve and $X$ is smooth near $C$.
\item[(1)] $f\colon Y \to X$ is a projective morphism.
\item[(2)] A general fiber of $f|_E \colon E \to C$ is smooth, connected,  and contained in the regular locus of $E$.
\item[(3)] The singular locus of $Y$ is not dominant over $C$.
\item[(4)] $-E$ is a $f$-nef $\QQ$-Cartier divisor  on $Y$ and $f$-ample outside finitely many points of $X$.
\item[(5)] $-K_Y$ is a $\QQ$-Cartier divisor  on $Y$ and $f$-ample outside finitely many points of $X$.
\end{itemize}
Then, the contraction $f\colon Y \to X$ coincides with the blow-up $p\colon {\rm{Bl}}_C X \to X$ of $X$ along $C$.
\end{lemm}
\begin{proof}
The problem is local in $X$.  Thus, we may assume that $X$ is smooth.
We first reduce the problem to show that $f\colon Y \to X$ coincides with
$p\colon {\rm{Bl}}_C X \to X$ outside finitely many points of $X$.
If $f\colon Y \to X$ coincides with the blow-up outside finitely many points of $X$,
then we can find
a bimeromorphic map
$$\text{
$\pi\colon Y \dashrightarrow {\rm{Bl}}_C X$ over $C $ that is isomorphic in codimension $1$
}
$$
by noting that $f\colon E \to C$ is an equidimensional fibration.
We conclude that $\pi\colon Y \dashrightarrow {\rm{Bl}}_C X$  is actually an isomorphism
by checking the assumptions of \cite[Remark 6.37, Lemma 6.39]{KM98} are satisfied.
For any $f$-ample Cartier divisor $A$ on $Y$,
the Weil divisor $\pi_{*}A$ corresponding to  $A$ via $\pi\colon Y \dashrightarrow {\rm{Bl}}_C X $
is a Cartier divisor on ${\rm{Bl}}_C X$ (since ${\rm{Bl}}_C X$ is smooth).
For the $p$-ample exceptional divisor $-B$ on ${\rm{Bl}}_C X$,
the Weil divisor $\pi^{-1}_{*}B$  is
$-E$, which is $f$-nef and $\QQ$-Cartier by assumption.
Thus, \cite[Lemma 6.39]{KM98}  shows that  $\pi\colon Y \dashrightarrow {\rm{Bl}}_C X$ is an isomorphism.

We finally show that $f\colon Y \to X$ coincides with the blow-up of $C$ outside finitely many points of $X$.
By Condition (4),
it is sufficient to show that for any $d \geq 0$,
$$f_*\OX(-dE)=\mathcal{I}_C^d$$
over the Zariski open set where $-E$ is $f$-ample.
We may assume that $Y$ is smooth by  Condition (3).
Furthermore, we may assume that
the restriction $f|_E\colon E \to C$ (on a non-empty Zariski open set of $E$)
is a smooth morphism between smooth spaces.
Since $\OX_E(-dE)=K_{E}-K_Y|_E-(d+1)E|_E$ holds and $-K_Y$ is $f$-ample,
the relative Kodaira vanishing theorem shows that
$$
R^if_*(\OX_E(-dE))=0   \text{ for any }i>0, d \geq 0,
$$
which finishes the proof by \cite[Lemma (3.32)]{Mor82}.
\end{proof}

We finally check that Theorem \ref{theo-main} follows from Theorem \ref{theo-non-proj}
by using almost the same arguments as in \cite[Theorem 1.4]{CH19}.

\begin{prop}\label{prop-main}
Theorem \ref{theo-main} follows from Theorem \ref{theo-non-proj}.
\end{prop}
\begin{proof}
Let $X$ be a non-projective compact K\"ahler three-fold with nef anti-canonical bundle.
We first show that $X$ admits a locally trivial MRC fibration $X \to Y$ onto a smooth surface with $c_1(Y)=0$.
The manifold $X$ admits a finite \'etale cover $X'$, one of the lists in Theorem \ref{theo-non-proj}.
If $c_{1}(X')=0$, then $c_{1}(X)=0$; hence we may assume that $X'$ admits a non-trivial locally constant MRC fibration.
Let $\mathcal{F}\subset T_X$ be the (unique) saturated integrable subsheaf such that
the $\mathcal{F}$-leaf through a very general point $x \in X$ is a fiber of the MRC-fibration.
By Theorem \ref{theo-non-proj}, the sheaf $\mathcal{F}$ is a regular foliation that is invariant under passing to finite \'etale covers.
Hence, by  \cite[Corollary 2.11]{Hor07},
there exists a smooth morphism $\varphi\colon  X \to Y$ such that $T_{X/Y}=\mathcal{F}$.
Since the general fiber is $\PP^1$, the fibration $\varphi \colon  X \to Y$ is locally trivial by the Firscher-Grauert theorem.
By Proposition \ref{push-forward-formula}, the line bundle $-K_Y$ is pseudo-effective.
Since $Y$ is not uniruled, we have $c_1(Y)=0$.

The fibration $\varphi\colon X \to Y$ is the MFS of $X$; hence $-K_{X}$ is $\varphi$-ample line bundle.
By Theorem \ref{theo-flat}, since $Y$ is smooth, we can find a $\varphi$-ample line bundle $B$ such that
$\varphi_{*}(pB)$ is numerically flat for $1 \ll p  \in \mathbb{Z}_{+}$.
Note that the numerically flatness follows from \cite[Main Theorem]{Wu22b}. 
Then, we see that  $\varphi \colon  X \to Y$  is locally constant by \cite[Proposition 2.5]{MW}.
Since $Y$ is a compact K\"ahler,  Condition (2) in \cite[Proposition 2.5]{MW} is automatically satisfied
(see \cite[Remark 2.6 (b)]{MW}). 
\end{proof}





\end{document}